\documentclass[reqno]{amsproc}
\usepackage{mathrsfs,amscd,amssymb,amsthm,amsmath,psfig,subfigure}
\newtheorem{theorem}{Theorem}[section]
\newtheorem{lemma}[theorem]{Lemma}

\newtheorem{thm}[theorem]{Theorem}
\newtheorem{prop}[theorem]{Proposition}
\newtheorem{cor}[theorem]{Corollary}

\theoremstyle{definition}

\theoremstyle{remark}
\newtheorem{rmk}[theorem]{Remark}
\newtheorem{exmp}[theorem]{Example}

\numberwithin{equation}{section}

\usepackage{bm}

\def\rk{\mbox{\scriptsize rank\,}}
\def\Rk{{\rm rank\,}}

\def\Ker{{\rm Ker\,}}
\def\Im{{\rm Im\,}}

\def\mod{{\rm mod\;}}
\def\Mod{{\rm Mod}}
\def\Tor{{\rm Tor}}

\def\CL{\textsc{cl}}
\def\TN{\textsc{tn}}

\def\sp{\hspace{1ex}}
\def\sp{\hspace{0.3cm}}

\begin{document}

\title[Orientations, polytopes, and subgroup arrangements]
{Orientations, lattice polytopes, and group arrangements I: Chromatic and
tension polynomials of graphs}

\author{Beifang Chen}
\address{Department of Mathematics,
Hong Kong University of Science and Technology,
Clear Water Bay, Kowloon, Hong Kong}
\curraddr{}

\email{mabfchen@ust.hk}
\thanks{The research is supported by the RGC Competitive Earmarked Research Grant 600703}


\date{28 October 2005}



\begin{abstract}
This is the first one of a series of papers on association of
orientations, lattice polytopes, and abelian group arrangements to
graphs. The purpose is to interpret the integral and modular tension
polynomials of graphs at zero and negative integers. The whole
exposition is put under the framework of subgroup arrangements and
the application of Ehrhart polynomials. Such viewpoint leads to the
following main results of the paper: (i) the reciprocity law for
integral tension polynomials; (ii) the reciprocity law for modular
tension polynomials; and (iii) a new interpretation for the value of
the Tutte polynomial $T(G;x,y)$ of a graph $G$ at $(1,0)$ as the
number of cut-equivalence classes of acyclic orientations on $G$.
\end{abstract}

\maketitle

\section{Introduction}

The chromatic polynomial $\chi(G,q)$ of a graph $G$, which was
introduced by Birkhoff \cite{Birkhoff1} as the number of proper
colorings of $G$ with $q$ colors, is one of the most mysterious but
fruitful polynomials in graph theory. It is the root of many other
important polynomials such as the characteristic polynomial for
graded posets and the Tutte polynomial for matroids; see
\cite{Bollobas1, Brylawski-Oxley1, Rota1, Stanley2}. Stimulated by
the Four Color Conjecture (FCC) of the time and the work of Rota
{\it et al} on the foundations of combinatorics, a flood of research
on such polynomials has continued for about three decades in
connection with graphs, matroids, posets, and simplicial complexes,
and the wave is still in its trend in connection with convexity,
toric geometry, Ehrhart polynomials, and topological invariants.

The key stimulation behind the research of characteristic
polynomials and their analogs is the geometric idea, which dates
back to Veblen's equivalent formulation of FCC \cite{Veblen1} as
solving a modular system of linear inequalities over the field
$GF(4)$. Such a geometric idea is implicit in the work of Crapo and
Rota \cite{Crapo-Rota1}, Greene and Zaslavsky
\cite{Greene-Zaslvsky1}, Kung \cite{Kung1}, Zaslavsky
\cite{Zaslavsky1}, and many others. More recent activities include
the work of Athanasiadis \cite{Athanasiadis1}, Bj\"orner and Ekedahl
\cite{Bjorner-Ekedahl1}, Chen \cite{Chen-Characteristic-Polynomial},
Ehrenborg and Readdy \cite{Ehrenborg-Readdy1}, Reiner
\cite{Reiner1}, and Sagan \cite{Sagan1}. To continue this geometric
trend, we propose to study systematically the chromatic polynomials,
flow polynomials, characteristic polynomials, and Tutte polynomials
in a series of papers by associating to graphs with orientations,
lattice polytopes, and abelian subgroup arrangements, and then
applying the theory of Ehrhart polynomials. The present paper is the
first one of a series of papers on such investigations,
concentrating on chromatic and tension polynomials. The main task is
to generalize Stanley's interpretation \cite{Stanley1} of chromatic
polynomials at negative integers to integral and modular tension
polynomials introduced by Kochol \cite{Kochol1}. Our approach is to
directly introduce group arrangements, lattice polytopes, and
equivalence relations on orientations of graphs in a rigorous
systematical way. We follow the books \cite{Bollobas1, Orlik-Terao1,
Stanley2, Ziegler1} for various notions on graphs, polytopes, and
hyperplane arrangements. For Ehrhart polynomials we refer to the
papers \cite{Chen-Lattice-Points, Chen-Ehrhart-Polynomial}. The
whole exposition is self-contained.

Let $G=(V,E)$ be a graph with possible loops and multiple edges; we
write $V=V(G)$ and $E=E(G)$. An {\em orientation} of $G$ is an
assignment that each edge of $G$ is assigned an arrow. If an edge is
incident with two vertices, there are two ways to assign an arrow;
if it is incident with one vertex, there is only one way to assign
an arrow. A graph with an orientation is called a {\em digraph}. A
{\em circuit} is a connected graph having degree 2 at every vertex;
a {\em direction} of a circuit is an orientation such that there is
an arrow in and an arrow out at every vertex; a circuit with a
direction is called a {\em directed circuit}. An orientation is said
to be {\em acyclic} if it has no directed circuit. We denote by
$\mathcal O(G)$ the set of all orientations of $G$, and by
$A\mathcal O(G)$ the set of all acyclic orientations. For positive
integers $q$, we define the counting function
\[
\bar\chi(G,q):=\#\big\{(\varepsilon,f)\:|\:\varepsilon\in A\mathcal
O(G), f\in\bigl([1,q]\cap{\Bbb Z}\bigr)^V, f(u)\geq f(v)
\;\mbox{for}\; u\xrightarrow{\varepsilon} v\big\}.
\]
For each acyclic orientation $\varepsilon\in A\mathcal O(G)$, let
$\bar\Delta^+_\CL(G,\varepsilon)$ denote the 0-1 polytope in the
Euclidean space ${\Bbb R}^V$ whose vertices are vectors
$f:V(G)\rightarrow\{0,1\}$ such that
\[
f(u)\geq f(v) \sp \mbox{for all}\sp u\xrightarrow{\varepsilon}v.
\]
We call $\bar\Delta^+_\CL(G,\varepsilon)$ the {\em coloring
polytope} and its interior $\Delta^+_\CL(G,\varepsilon)$ the {\em
open coloring polytope} of the digraph $(G,\varepsilon)$. Let
\[
\chi(G,\varepsilon;t) =L\bigl(\Delta^+_\CL(G,\varepsilon),t+1\bigr)
\]
denote the Ehrhart polynomial of $\Delta^+_\CL(G,\varepsilon)$ in
variable $t+1$, and let
\[
\bar\chi(G,\varepsilon;t)=L\bigl(\bar\Delta^+_\CL(G,\varepsilon),t-1\bigr)
\]
denote the Ehrhart polynomial of $\bar\Delta^+_\CL(G,\varepsilon)$
in variable $t-1$. The following Theorem~\ref{CP} is a reformulation
of Stanley's interpretation for chromatic polynomials at negative
integers.

\begin{thm}[Stanley \cite{Stanley1}]\label{CP}
Let $G$ be a loopless graph with possible multiple edges. Then the
chromatic polynomials $\chi(G,t)$ and $\bar\chi(G,t)$ can be written
as
\begin{eqnarray}
\chi(G,t)
&=& \sum_{\varepsilon\in A{\mathcal O}(G)} \chi(G,\varepsilon;t),\label{CP-Sum-Open} \\
\bar\chi(G,t) &=& \sum_{\varepsilon\in A{\mathcal
O}(G)}\bar\chi(G,\varepsilon;t). \label{CP-Sum-Closed}
\end{eqnarray}
Moreover, for each orientation $\varepsilon\in A{\mathcal O}(G)$, we have the
{\em reciprocity law}
\begin{equation}\label{Epsilon-CP-RL}
\chi(G,\varepsilon;-t)=(-1)^{|V|}\bar\chi(G,\varepsilon;t),
\end{equation}
\begin{equation}\label{CP-RL}
\chi(G,-t)=(-1)^{|V|}\bar\chi(G,t).
\end{equation}
In particular, $\bar\chi(G,\varepsilon;1)=(-1)^{|V|}\chi(G,\varepsilon;-1)=1$,
and $|\chi(G,-1)|$ counts the number of acyclic orientations of $G$.
\end{thm}

Let $(H_i,\varepsilon_i)$ be oriented subgraphs of $G$, $i=1,2$. We
introduce a {\em coupling} of $\varepsilon_1$ and $\varepsilon_2$,
defined as a function
$[\varepsilon_1,\varepsilon_2]:E\rightarrow{\Bbb Z}$ by
\begin{equation}
[\varepsilon_1,\varepsilon_2](x)=\left\{\begin{array}{rl}
1 & \mbox{if} \sp \varepsilon_1(x)=\varepsilon_2(x),\; x\in E(H_1)\cap E(H_2), \\
-1 & \mbox{if} \sp  \varepsilon_1(x)\neq \varepsilon_2(x),\; x\in E(H_1)\cap E(H_2), \\
0 & \mbox{otherwise}.
\end{array}\right.
\end{equation}

Let $A$ be an abelian group. A {\em tension} of $(G,\varepsilon)$ with values
in $A$ is a function $f:E(G)\rightarrow A$ such that for any directed circuit
$(C,\varepsilon_{\textsc c})$,
\begin{equation}\label{Tension-Definition}
\sum_{x\in C}[\varepsilon,\varepsilon_{\textsc c}](x)f(x)=0.
\end{equation}
We denote by $T(G,\varepsilon; A)$ the abelian group of all tensions
of $(G,\varepsilon)$ with values in $A$. Tensions with values in
$\Bbb R$ (or $\Bbb Z$) are called {\em real-valued} (or {\em
integral}). Let $q$ be a positive integer. A real-valued tension $f$
is called a {\em $q$-tension} if $|f(x)|<q$ for all $x\in E(G)$. We
denote by $\tau(G,q)$ the number of nowhere-zero tensions of
$(G,\varepsilon)$ with values in $A$ whose order is $q$, and by
$\tau_{\Bbb Z}(G,q)$ the number of nowhere-zero integral
$q$-tensions. It turns out that both $\tau(G,q)$ and $\tau_{\Bbb
Z}(G,q)$ are independent of the chosen orientation $\varepsilon$ and
the group structure of $A$. Moreover, $\tau(G,q)$ and $\tau_{\Bbb
Z}(G,q)$ are polynomial functions of positive integers $q$, having
degree $r(G)=|V|-k(G)$, where $k(G)$ is the number of connected
components of $G$. We call $\tau(G,q)$ and $\tau_{\Bbb Z}(G,q)$ the
{\em modular tension polynomial} and the {\em integral tension
polynomial} of $G$, respectively.

To interpret the values of the integral tension polynomial
$\tau_{\Bbb Z}(G,t)$ at negative integers, let us introduce the
counting function
\[
\bar\tau_{\Bbb Z}(G,q):=\#\big\{(\varepsilon,f)\:|\:\varepsilon\in
A\mathcal O(G), f\in T(G,\varepsilon;{\Bbb Z}), 0\leq f\leq q\big\}.
\]
For each $\varepsilon\in A\mathcal O(G)$, let
$\bar\Delta^+_{\TN}(G,\varepsilon)$ be the 0-1 polytope in the
tension vector space $T(G,\varepsilon;{\Bbb R})$ whose vertices are
vectors $f\in T(G,\varepsilon;{\Bbb R})$ such that $f(x)=0$ or $1$.
We call $\bar\Delta^+_{\TN}(G,\varepsilon)$ the {\em tension
polytope} and its interior $\Delta^+_{\TN}(G,\varepsilon)$ the {\em
open tension polytope} of digraph $(G,\varepsilon)$. Let $\tau_{\Bbb
Z}(G,\varepsilon;t)$ and $\bar\tau_{\Bbb Z}(G,\varepsilon;t)$ denote
the Ehrhart polynomials of $\Delta^+_{\TN}(G,\varepsilon)$ and
$\bar\Delta^+_{\TN}(G,\varepsilon)$, respectively. Our first main
result is the following theorem.

\begin{thm}\label{Integral-TNP}
Let $G$ be a loopless graph with possible multiple edges. Then $\tau_{\Bbb
Z}(G,q)$ and $\bar\tau_{\Bbb Z}(G,q)$ are polynomial functions of degree $r(G)$
in positive integers $q$, and can be written as
\begin{eqnarray}
\tau_{\Bbb Z}(G,t)
&=& \sum_{\varepsilon\in A{\mathcal O}(G)} \tau_{\Bbb Z}(G,\varepsilon;t),
\label{Integral-TNP-Sum-Open} \\
\bar\tau_{\Bbb Z}(G,t) &=& \sum_{\varepsilon\in A{\mathcal O}(G)}\bar\tau_{\Bbb
Z}(G,\varepsilon;t). \label{Integral-TNP-Sum-Closed}
\end{eqnarray}
Moreover, for each orientation $\varepsilon\in A{\mathcal O}(G)$, we have the
{\em reciprocity law}
\begin{eqnarray}
\tau_{\Bbb Z}(G,\varepsilon;-t) &=& (-1)^{r(G)}\bar\tau_{\Bbb Z}(G,\varepsilon;t),
\label{Epsilon-Integral-TNP-RL}\\
\tau_{\Bbb Z}(G,-t) &=& (-1)^{r(G)}\bar\tau_{\Bbb Z}(G,t). \label{TNP-RL}
\end{eqnarray}
In particular, $\bar\tau_{\Bbb Z}(G,\varepsilon;0)=(-1)^{r(G)}\tau_{\Bbb
Z}(G,\varepsilon;0)=1$; $|\tau_{\Bbb Z}(G,0)|$ counts the number of acyclic
orientations of $G$.
\end{thm}

To interpret the values of the modular tension polynomial
$\tau(G,t)$ at negative integers, let us define an equivalence
relation $\sim$ on $\mathcal O(G)$. Recall that a {\em cut} of $G$
is a subset $[S,T]\subseteq E(G)$, where $\{S,T\}$ is a partition of
$V$ and $[S,T]$ is the set of all edges between $S$ and $T$. Let
$(G,\varepsilon)$ be a digraph. If $[S,T]$ is a cut, we denote by
$(S,T)_\varepsilon$ the set of all edges having arrows from $S$ to
$T$, and by $(T,S)_\varepsilon$ the set of all edges having arrows
from $T$ to $S$. A {\em bond} is a minimal cut. A cut $[S,T]$ of
$(G,\varepsilon)$ is said to be {\em directed} if either
$[S,T]=(S,T)_\varepsilon$ or $[S,T]=(T,S)_\varepsilon$, i.e., all
edges of $[S,T]$ have arrows from $S$ to $T$ or all have arrows from
$T$ to $S$. A minimal directed cut is called a {\em directed bond}.
A cut is said to be {\em oriented} if it can be decomposed into a
disjoint union of directed bonds. Two orientations
$\varepsilon,\varrho\in\mathcal O(G)$ are called {\em
cut-equivalent}, denoted $\varepsilon\sim\varrho$, if the edge set
\[
E(\varepsilon\neq\varrho): =\big\{x\in
E\:|\:\varepsilon(x)\neq\varrho(x)\big\}
\]
is empty or is an oriented cut with respect to either orientation
$\varepsilon$ or $\varrho$. It turns out that $\sim$ is indeed an
equivalence relation on $\mathcal O(G)$. Moreover, if
$\varepsilon\sim\varrho$ and $\varepsilon$ is acyclic, then
$\varrho$ is also acyclic. So $\sim$ induces an equivalence relation
on the set $A\mathcal O(G)$ of acyclic orientations of $G$. Let
$[A\mathcal O(G)]$ denote the quotient set of cut-equivalence
classes of acyclic orientations of $G$. We introduce the counting
function
\[
\bar\tau(G,q): =\#\big\{([\varepsilon],f)\:|\:[\varepsilon]\in
[A\mathcal O(G)], f\in T(G,[\varepsilon];{\Bbb Z}), 0\leq f\leq
q\big\},
\]
where $T\bigl(G,[\varepsilon];{\Bbb Z}\bigr)=T(G,\varepsilon;{\Bbb Z})$ and
$\varepsilon\in[\varepsilon]$ is any representative. Our last main result is
the following theorem.

\begin{thm}\label{Modular-TNP-Decom-RL}
Let $G=(V,E)$ be a loopless graph with possible multiple edges. Then
$\tau(G,q)$ and $\bar\tau(G,q)$ are polynomial functions of degree $r(G)$ in
positive integers $q$, and can be written as
\begin{gather}
\tau(G,t) = \sum_{\varepsilon\in[A\mathcal O(G)]} \tau_{\Bbb
Z}(G,\varepsilon;t), \label{TGT}
\\
\bar\tau(G,t) = \sum_{\varepsilon\in[A\mathcal O(G)]} \bar\tau_{\Bbb
Z}(G,\varepsilon;t), \label{BTGT}
\end{gather}
and satisfy the {\em reciprocity law}
\begin{gather}
\tau(G,-t)=(-1)^{r(G)}\bar\tau(G,t).
\end{gather}
In particular, $(-1)^{r(G)}\tau(G,0)$ counts the number of cut-equivalence
classes of acyclic orientations of $G$.
\end{thm}

After searching literature on tension polynomials, we found that formulas
equivalent to (\ref{Integral-TNP-Sum-Open}) and (\ref{TGT}) were observed by
Kochol \cite{Kochol1} in different form with incomplete proof. It is
interesting that certain bounds were obtained in \cite{Kochol1} for the modular
and integral tension polynomials $\tau(G,q)$ and $\tau_{\Bbb Z}(G,q)$. However,
 the interpretations for the values of the polynomials at zero and negative integers
are not considered there.

\begin{cor}
Let $T(G;x,y)$ be the Tutte polynomial of a graph $G$. If $G$ has no loops,
then
\begin{equation}
T(G;t,0)=\bar\tau(G,t-1)=(-1)^{r(G)}\tau(G,1-t).
\end{equation}
Moreover, $T(G;1,0)=(-1)^{r(G)}\tau(G,0)=\bar\tau(G,0)$ counts the number of
cut-equivalence classes of acyclic orientations of $G$.
\end{cor}

At the end we remark that the field $GF(4)$ in Veblen's equivalent formulation
\cite{Veblen1} for FCC can be replaced by any abelian group of 4 elements,
i.e., every planar graph $G=(V,E)$ with $V=\{1,2,\ldots,n\}$ can be properly
colored with 4 colors if and only if the system of linear inequalities
\[
x_i-x_j\neq 0, \sp (i,j)\in E
\]
has a solution over an abelian group of 4 elements. For an arbitrary
graph $G$, the chromatic number of $G$ is the minimal order of an
abelian group such that the above system of linear inequalities has
a solution. The fact that the field can be replaced by any abelian
group is that graphs are torsion free or their torsion is 1.

\section{Valuations on abelian groups}

Let $\Omega$ be a finitely generated abelian group, not necessarily
finite and free. By a {\em flat} of $\Omega$ we mean a coset of a
subgroup of $\Omega$. Let $\Gamma\subseteq\Omega$ be a subgroup. Let
$\Tor(\Gamma)$ denote the torsion subgroup of $\Gamma$. The {\em
size} of $\Gamma$ is defined as
\[
|\Gamma|: =|\Tor(\Gamma)|\,t^{\rk(\Gamma)}.
\]
If $\Gamma$ is finite, then $|\Gamma|$ is the number of elements of
$\Gamma$. For each coset $F$ in $\Omega$, the {\em rank}, the {\em
torsion subgroup}, and the {\em size} of $F$ are defined as the
rank, the torsion subgroup, and the size of $\Gamma$, respectively.
For any subset $X\subseteq\Omega$ and $a\in\Omega$, we define the
translate
\[
a+X: =\big\{a+x\:|\: x\in X\big\}.
\]

Let ${\mathscr L}(\Omega)$ denote the set of all cosets in $\Omega$,
called the {\em semilattice} of cosets in $\Omega$. An {\em affine
subgroup arrangement} (or just {\em arrangement}) of $\Omega$ is a
collection $\mathcal A$ of finitely many flats of $\Omega$. Let
$\mathcal A$ be an arrangement of $\Omega$. Let ${\mathscr
L}({\mathcal A})$ be the set of all nonempty intersections of flats
in $\mathcal A$, called the {\em semilattice} of $\mathcal A$.
Notice that ${\mathscr L}({\mathcal A})$ is closed under
intersection. We introduce the {\em characteristic polynomial} of
$\mathcal A$ as
\[
\chi(\mathcal A,t) = \sum_{X\in\mathscr L(\mathcal A)}
\frac{\mu(X,\Omega)|\Tor(\Omega)|}{|\Tor(\Omega/\langle X\rangle)|}
\,t^{\rk\langle X\rangle},
\]
where $\mu$ is the M\"{o}bius function of the poset $\mathscr
L({\mathcal A})$, whose partial order is the set inclusion, and
$\langle X\rangle=\{x-y\:|\:x,y\in X\}$. Let ${\mathscr B}(\Omega)$
denote the Boolean algebra generated by ${\mathscr L}(\Omega)$,
i.e., every element of $\mathscr B(\Omega)$ is obtained from cosets
of $\Omega$ by taking unions, intersections, and complements
finitely many times.

A {\em valuation} on $\Omega$ with values in an abelian group $A$ is
a map $\nu:\mathscr B(\Omega)\rightarrow A$ such that
\begin{gather}
\nu(\emptyset)=0, \\
\nu(X\cup Y)=\nu(X)+\nu(Y)-\nu(X\cap Y)
\end{gather}
for $X,Y\in{\mathscr B}(\Omega)$. A valuation $\nu$ on $\Omega$ is
said to be {\em translate-invariant} if for any $a\in\Omega$ and
$X\in{\mathscr B}(\Omega)$,
\begin{gather}
\nu(a+X)=\nu(X);
\end{gather}
and $\nu$ is said to satisfy {\em productivity} if
\begin{gather}
\nu(A+B)=\nu(A)\nu(B)
\end{gather}
for subgroups $A,B\subset\Omega$ such that $A+B$ is a direct sum and
$A+B$ is a direct summand of $\Omega$.

Let $X\subseteq\Omega$ be a subset; its {\em indicator function}
$1_X:\Omega\rightarrow{\Bbb Z}$ is defined by
\[
1_X(x)=\left\{\begin{array}{ll} 1 & \mbox{if}\; x\in X,\\
0 & \mbox{if}\; x\not\in X.
\end{array}\right.
\]
Let $\mathscr F(\Omega)$ denote the abelian group of all functions
$f:\Omega\rightarrow{\Bbb Z}$ such that $f(\Omega)$ is a finite set
and $f^{-1}(n)\in{\mathscr B}(\Omega)$ for $n\in{\Bbb Z}$. Then each
such function can be written as a linear combination
\[
f=\sum a_i1_{F_i},
\]
where $F_i$ are cosets of $\Omega$, $a_i\in{\Bbb Z}$. The {\em integral} of $f$
with respect to a valuation $\nu$ is defined as
\[
\nu(f) = \int_\Omega f{\rm d}\nu(x) =\sum a_i\,\nu(F_i)
=\sum_{n\in{\Bbb Z}} n\,\nu\bigl(f^{-1}(n)\bigr).
\]
By Groemer's Extension Theorem \cite{Groemer1}, a set function
$\nu:\mathscr L(\Omega)\rightarrow{A}$ can be extended to a
valuation $\nu:\mathscr B(\Omega)\rightarrow{A}$ if and only if the
{\em inclusion-exclusion formula}
\begin{gather}\label{Inclusion-Exclusion}
\nu(L)= \sum_{k=1}^n(-1)^{k-1} \sum_{i_1<\cdots<i_k}
\nu\bigl(L_{i_1}\cap\cdots\cap L_{i_k}\bigr)
\end{gather}
is satisfied, provided that $L,L_1,\ldots,L_n\in\mathscr L(\Omega)$
and $L=L_1\cup\cdots\cup L_n$.

\begin{thm}
For any finitely generated abelian group $\Omega$, there is a unique
translate-invariant valuation $\lambda: {\mathscr
B}(\Omega)\rightarrow{\Bbb Q}[t]$ such that the productivity is
satisfied and
\[
\lambda(\Omega) =|\Tor(\Omega)|t^{\rk(\Omega)} =|\Omega|.
\]
Moreover, for any finite subset $S\subseteq\Omega$ and subgroup
$\Gamma\subseteq\Omega$, we have
\[
\lambda(S)=|S|, \sp \lambda(\Gamma)=\frac{|\Omega|}{|\Omega/\Gamma|}.
\]
\end{thm}
\begin{proof}
For a subgroup $\Gamma\subseteq\Omega$, let $Z_\Gamma$ be a free
subgroup of $\Gamma$ such that $\Gamma=\Tor(\Gamma)+Z_\Gamma$ is a
direct sum. Let $Z$ be a free subgroup of $\Omega$ such that
$Z_\Gamma\subseteq Z$ and $\Omega=T+Z$ is a direct sum, where
$T=\Tor(\Omega)$. Since $\Omega=\bigsqcup_{a\in T}\big(a+Z\big)$, by
additivity and productivity we have
$\lambda(\Omega)=|T|\lambda(Z)=\lambda(T)\lambda(Z)$. Thus
$\lambda(T)=|T|$ and $\lambda(Z)=t^{\rk(Z)}$. Since $\lambda$ is
translate-invariant, it follows that $\lambda(\{0\})=1$. So
$\lambda(S)=|S|$ for any finite subset $S\subseteq\Omega$.

Let $a_1,\ldots,a_k$ be a collection of generators of $Z$. Then
\[
{\Bbb Z}a_1+{\Bbb Z}a_2 ={\Bbb Z}a_1+{\Bbb Z}(a_1+a_2) ={\Bbb Z}(a_1+a_2)+{\Bbb
Z}a_2
\]
and the sums are direct. Again by the productivity we have
\[
\lambda({\Bbb Z}a_1+{\Bbb Z}a_2) =\lambda({\Bbb Z}a_1)\lambda\big({\Bbb
Z}(a_1+a_2)\big) =\lambda\big({\Bbb Z}(a_1+a_2)\big)\lambda({\Bbb Z}a_2).
\]
Thus $\lambda({\Bbb Z}a_1)=\lambda({\Bbb Z}a_2)$. Since
$\lambda(Z)=t^{\rk(Z)}$, it follows that $\lambda({\Bbb Z}a_1)=t$.
So $\lambda(H)=t^{\rk(H)}$ for any direct summand $H$ of $Z$.

Let $\tilde Z_\Gamma =\{z\in Z\:|\: mz\in Z_\Gamma\;\mbox{for some
$m$}\}$. Then $\tilde Z_\Gamma$ is a direct summand of $\Omega$.
Since $\tilde Z_\Gamma=\bigsqcup_{[a]\in{\tilde
Z_\Gamma/Z_\Gamma}}(a+Z_\Gamma)$, we have $\lambda(\tilde
Z_\Gamma)=|\tilde
Z_\Gamma/Z_\Gamma|\lambda(Z_\Gamma)=t^{\rk(\Gamma)}$. Hence
\[
\lambda(\Gamma)=\lambda\big(\Tor(\Gamma)\big)\lambda(Z_\Gamma)
=\frac{|\Tor(\Gamma)|}{|\tilde Z_\Gamma/Z_\Gamma|}\,
t^{\rk(\Gamma)}.
\]
Notice that
\begin{eqnarray*}
|\Omega| &=& |\Tor(\Gamma)|\,|\Tor(\Omega)/\Tor(\Gamma)|\,t^{\rk(\Omega)},\\
|\Omega/\Gamma| &=& |\Tor(\Omega)/\Tor(\Gamma)|\,|\tilde
Z_\Gamma/Z_\Gamma|\,t^{\rk(\Omega)-\rk(\Gamma)}.
\end{eqnarray*}
It follows that  $\lambda(\Gamma)|=|\Omega)|/|\Omega/\Gamma|$.

Next we show that the inclusion-exclusion formula
(\ref{Inclusion-Exclusion}) is satisfied. We proceed by induction on
$n$, the number of cosets in (\ref{Inclusion-Exclusion}). For $n=1$,
it is obviously true. Assume it is true for the case $n-1$, and
consider the case $n$. Let $L,L_1,\ldots,L_n\in\mathscr L(\Omega)$,
$L=L_1\cup\cdots\cup L_n$. Without loss of generality, we may assume
that $L$ contains the zero element $0$, i.e., $L$ is a subgroup.
Notice that one of the cosets $L_1,\ldots,L_n$ must have the same
rank as $L$, say, $\Rk(L_n)=\Rk(L)$. Fix an element $a\in L_n$; we
set
\[
\Gamma:=L_n-a=\big\{x-a\:|\:x\in L_n\big\}.
\]
Then $\Gamma$ is a subgroup, and $L_n$ is a coset of $\Gamma$.
Clearly, $\Gamma$ is a subgroup of $L$ with the same rank. Thus the
quotient group $L/\Gamma$, consisting of all cosets of $\Gamma$ in
$L$, is a finite set. Notice that all cosets $F\in L/\Gamma$, except
$L_n$, are contained in $L_1\cup\cdots\cup L_{n-1}$. Hence
$F=(L_1\cap F)\cup\cdots\cup (L_{n-1}\cap F)$. By induction, for
$F\in L/\Gamma$ such that $F\neq L_n$, we have
\[
\lambda(F) = \sum_{I\subseteq[n-1],I\neq\emptyset} (-1)^{|I|-1}
\lambda\Biggl(\bigcap_{i\in I} L_i\cap F\Biggr).
\]
On the other hand, it is routine to check that for any coset $L'$ of $L$,
\[
\lambda(L')=\sum_{F\in L/\Gamma}\lambda(L'\cap F).
\]
In particular, for the cosets $L_I=\bigcap_{i\in I}L_i$ of $L$ with nonempty
$I\subseteq[n]$, we have
\[
\lambda(L_I) = \sum_{F\in L/\Gamma}\lambda(L_I\cap F).
\]
Now the right-hand side of (\ref{Inclusion-Exclusion}) can be written as
\begin{eqnarray*}
{\rm RHS} &=& \sum_{I\subseteq[n],I\neq\emptyset} (-1)^{|I|-1} \sum_{F\in
L/\Gamma} \lambda(L_I\cap F) \\
&=& \sum_{F\in L/\Gamma} \sum_{I\subseteq[n]\atop I\neq\emptyset} (-1)^{|I|-1}
\lambda(L_I\cap F) \\
&=& \left\{\sum_{F\in L/\Gamma\atop F=L_n} + \sum_{F\in L/\Gamma\atop F\neq
L_n}\right\} \left\{\sum_{I\subseteq[n]\atop n\in I} + \sum_{I\subseteq[n]\atop
I\neq\emptyset,n\not\in I}\right\}.
\end{eqnarray*}
Clearly, the RHS is decomposed into the following four sums:
\begin{align}
\sum_{F\in L/\Gamma\atop F=L_n} \sum_{I\subseteq[n]\atop n\in I} & =
\sum_{I\subseteq[n-1]} (-1)^{|I|} \lambda(L_I\cap L_n) \nonumber\\
& =  \lambda(L_n)+\sum_{I\subseteq[n-1],\,I\neq\emptyset} (-1)^{|I|}
\lambda(L_I\cap L_n);\nonumber
\end{align}
\begin{align}
 \sum_{F\in L/\Gamma\atop F\neq L_n} \sum_{I\subseteq[n]\atop n\in I}
& = 0 \sp \mbox{(since $L_I\cap F=\emptyset$)}; \nonumber
\end{align}
\begin{align}
\sum_{F\in L/\Gamma\atop F=L_n} \sum_{I\subseteq[n]\atop
I\neq\emptyset,n\not\in I} & = \sum_{I\subseteq[n-1],\,I\neq\emptyset}
(-1)^{|I|-1} \lambda\bigl(L_I\cap L_n\bigr); \nonumber
\end{align}
\begin{align}
\sum_{F\in L/\Gamma\atop F\neq L_n} \sum_{I\subseteq[n]\atop
I\neq\emptyset,n\not\in I} & = \sum_{F\in L/\Gamma\atop F\neq L_n}
\sum_{I\subseteq[n-1]\atop
I\neq\emptyset} (-1)^{|I|-1} \lambda\bigl(L_I\cap F\bigr) \nonumber \\
& = \sum_{F\in L/\Gamma\atop F\neq L_n} \lambda(F) =\lambda(L)-\lambda(L_n).
\nonumber
\end{align}
It follows that the right-hand side of (\ref{Inclusion-Exclusion}) equals
$\lambda(L)$, as desired. The uniqueness is obvious.
\end{proof}

\begin{thm}\label{Chracteristic-Formula-Thm}
Let $\mathcal A$ be an affine subgroup arrangement of a finitely generated
abelian group $\Omega$. Then
\begin{gather}\label{Chracteristic-Formula}
\chi(\mathcal A,t)=\lambda\Biggl(\Omega-\bigcup_{F\in\mathcal A} F\Biggr).
\end{gather}
\end{thm}
\begin{proof}
For each coset $Y\in\mathscr L(\mathcal A)$, let $Y^0=Y-\bigcup_{X<Y}X$. Then
$\bigl\{X^0\:|\:X\in\mathscr L(\mathcal A)\bigr\}$ is a collection of disjoint
subsets. Since $Y=\bigsqcup_{X\leq Y}X^0$ for all $Y\in{\mathscr L}({\mathcal
A})$, we have
\[
1_Y=\sum_{X\leq Y} 1_{X^0}, \sp Y\in \mathscr L(\mathcal A).
\]
Applying the M\"{o}bius inversion, we obtain
\[
1_{Y^0} = \sum_{X\leq Y}\mu(X,Y)\,1_X, \sp Y\in \mathscr L(\mathcal A).
\]
In particular, for $Y=\Omega$, we have $\Omega^0=\Omega-\bigcup_{X\in\mathcal
A} X$, and
\[
1_{\Omega^0}= \sum_{X\in\mathscr L(\mathcal A)} \mu(X,\Omega)\,1_X.
\]
Applying the valuation $\lambda$ to both sides, we see that
\begin{eqnarray*}
\lambda(\Omega^0) &=& \sum_{X\in\mathscr L(\mathcal A)} \mu(X,\Omega)
\lambda(X) \\
&=& \sum_{X\in\mathscr L(\mathcal A)}
\mu(X,\Omega)\cdot\frac{|\Omega|}{|\Omega/\langle X\rangle|} \\
&=& \sum_{X\in\mathscr L(\mathcal A)}
\frac{\mu(X,\Omega)|\Tor(\Omega)|}{|\Tor(\Omega/\langle X\rangle)|} \,t^{\rk\langle X\rangle} \\
&=& \chi(\mathcal A,t).
\end{eqnarray*}
\end{proof}

\begin{rmk}
When $\Omega$ is a finite-dimensional vector space, the unique
translate-invariant valuation $\lambda$ was obtained by Ehrenborg
and Readdy \cite{Ehrenborg-Readdy1}; see also Chen
\cite{Chen-Infinite-Subspace-Arrangement} for the generalization to
infinite-dimensional vector spaces.
\end{rmk}

\section{Modular chromatic and tension polynomials}

Let $G=(V,E)$ be a loopless graph with possible multiple edges. Let
$A$ be an abelian group. A {\em coloring} of $G$ with the color set
$A$ is a function $f:V\rightarrow A$; $f$ is said to be {\em proper}
if $f(u)\neq f(v)$ for any adjacent vertices $u,v$. We denote by
$K(G,A)$ the set of all colorings of $G$, and by $K_{\rm nz}(G,A)$
the set of all proper colorings. If $|A|=q$ is finite, it is
well-known that the counting function
\begin{gather}\label{Chromatic-Poly}
\chi(G,q):=\big|K_{\rm nz}(G,A)\big|
\end{gather}
is a polynomial function of $q$, depending only on the order of $A$, not on the
group structure; $\chi(G,t)$ is called the {\em chromatic polynomial} of $G$.

Let $\varepsilon$ be an orientation of $G$. We denote by $T(G,\varepsilon;A)$
the abelian group of all tensions of the digraph $(G,\varepsilon)$ with values
in $A$, called the {\em tension group} of $(G,\varepsilon)$, and by $T_{\rm
nz}(G,\varepsilon;A)$ the set of all nowhere-zero tensions. If $A$ is finite,
we shall see that $|T(G,\varepsilon;A)|$ and $|T_{\rm nz}(G,\varepsilon;A)|$
depend only on the order of $A$, but not on the abelian group structure. So,
for $|A|=q$, we define the counting function
\[
\tau(G,q):=\bigl|T_{\rm nz}(G,\varepsilon;A)\bigr|.
\]
We shall see that $\tau(G,q)$ is a polynomial function of positive integers
$q=|A|$, and is independent of the orientation $\varepsilon$ and the abelian
group structure of $A$.

Note that a coloring of $G$ may be viewed as a potential function on
$G$. There is natural {\em difference operator}
$\delta:A^V\rightarrow A^E$ defined by
\begin{gather}
(\delta f)(x)=f(u)-f(v),
\end{gather}
where $x=uv$ is an edge with the orientation
$u\xrightarrow{\varepsilon}v$. When a graph is viewed as a
1-dimensional simplicial complex, the difference operator $\delta$
is known as {\em coboundary operator}, and the tension group
$T(G,\varepsilon;A)$ is known as {\em cohomology group}. The
following Proposition~\ref{Potential-Tension} and its corollary
state the relation between colorings and tensions; see
\cite{Berge,Bondy-Murty1}.

\begin{prop}\label{Potential-Tension}
\begin{enumerate}
\item[(a)] $\Im\delta=T(G,\varepsilon;A)$.

\item[(b)] $\delta:K(G,A)\rightarrow T(G,\varepsilon;A)$ is a group homomorphism
with $\Ker\delta\simeq A^{k(G)}$, where $k(G)$ is the number of
connected components of $G$.

\item[(c)] The restriction $\delta:K_{\rm nz}(G,A)\rightarrow T_{\rm
nz}(G,\varepsilon;A)$ is well defined.
\end{enumerate}
\end{prop}
\begin{proof}
(a) Let $\tilde f\in T(G,\varepsilon;A)$. We construct a coloring $f\in K(G,A)$
as follows. Let $G_0$ be a connected component of $G$. Fix a vertex $v_0\in
G_0$ and an element $a\in A$. For each $v\in V(G_0)$, let $P=v_0v_1\ldots v_n$
be a shortest path from $v_0$ to $v=v_n$, and let $\varrho$ be an orientation
of $P$ such that $v_{i-1}\xrightarrow{\varrho}v_i$. We define $f(v_0)=a$, and
\[
f(v) = a-\sum_{i=1}^n[\varepsilon,\varrho](v_{i-1},v_i)\tilde f(v_{i-1},v_i).
\]
It is routine to check that $f$ is well-defined. Moreover, $f$ is proper if and
only if $\tilde f$ is nowhere-zero. Since $a$ is arbitrarily given for each
component, we see that $\Ker(d)\simeq A^{k(G)}$.

(b) and (c) are trivial by the construction of $f$ from $\tilde f$.
\end{proof}

\begin{cor}
The chromatic polynomial $\chi(G,t)$ and the modular tension polynomial
$\tau(G,t)$ are related by
\begin{gather}\label{Chromatic-Tension-Polynomial-Relation}
\chi(G,t)=t^{k(G)}\tau(G,t),
\end{gather}
where $k(G)$ is the number of connected components of $G$.
\end{cor}
\begin{proof}
It follows from (b) and (c) of Proposition~\ref{Potential-Tension}.
\end{proof}

Let $F$ be a spanning forest of $G$. For each edge $e\in F$, let
$B_e$ be the unique bond, whose edge set consists of the edge $e$
and the edges $x$ of $F^c$ such that the circuit of $F\cup x$
contains the edge $e$. Let $\varepsilon_e$ be the direction of $B_e$
such that $\varepsilon_e(e)=\varepsilon(e)$. Then the function
$[\varepsilon,\varepsilon_e]$ is a tension of $(G,\varepsilon)$. For
any $f\in A^E$, the function $f_e:=f(e)[\varepsilon,\varepsilon_e]$
is a tension of $(G,\varepsilon)$.

\begin{lemma}\label{TN-Solution-Lemmma}
Let $F$ be a spanning forest of $G$. Then the linear system
$(\ref{Tension-Definition})$, whose equations are indexed by all circuits, is
equivalent to the linear system
\begin{equation}\label{TN-Solution-Equation}
f(x) = \sum_{e\in F} [\varepsilon,\varepsilon_e](x) f(e), \sp x\in F^c
\end{equation}
with equations indexed by edges $x\in F^c$. In other words,
$(\ref{TN-Solution-Equation})$ solves the linear system
$(\ref{Tension-Definition})$ with $|F|$ free variables.
\end{lemma}
\begin{proof}
For $f\in A^E$, define $f':=f-\sum_{e\in F}f_e$. Then $f'$ is a tension of
$(G,\varepsilon)$ if and only if $f$ is a tension. Since
$\varepsilon_e(e)=\varepsilon(e)$, we have $f_e(e)=f(e)$. For any $x\in F$ but
$x\neq e$, we have $f_e(x)=0$ as $x\not\in B_e$. This implies that $f'\equiv 0$
on $F$.

For any $y\in F^c$, let $C_y$ be the unique circuit contained in $F\cup y$. Let
$\varepsilon_y$ be a direction of $C_y$. If $f'$ is a tension, then by
definition, $\sum_{x\in C_y}[\varepsilon,\varepsilon_y](x)f'(x)=0$; as
$f'(x)=0$ for any $x\in C_y$ but $y$, we have $f'(y)=0$. This means that
$f'\equiv 0$ on $E(G)$. Thus $f=\sum_{e\in F}f_e$.
\end{proof}

\begin{cor}\label{TGEGamma}
The tension group $T(G,\varepsilon;A)$ is isomorphic to the $A$-free abelian
group $A^{r(G)}$, where $r(G)=|V|-k(G)$. Moreover, for any subset $X\subseteq
E(G)$, the abelian group
\[
T_X(G,\varepsilon;A): =\big\{f\in T(G,\varepsilon;A)\:|\:
f(x)=0\;\mbox{\em for}\; x\in X\big\}
\]
is isomorphic to $A^{r(G)-r\langle X\rangle}$, where $\langle X\rangle=(V,X)$.
In particular, if $|A|=q$ is finite, then
\begin{equation}\label{TXGEG}
|T(G,\varepsilon;A)|=q^{r(G)}, \sp |T_X(G,\varepsilon;A)|=q^{r(G)-r\langle
X\rangle}.
\end{equation}
\end{cor}
\begin{proof} Let $F_X$ be a spanning forest of $X$ in the sense that each component
of $F$ is a spanning tree of a component of $G$. Then $F_X$ can be
extended to a forest $F$ of $G$. By Lemma~\ref{TN-Solution-Lemmma},
we see that $T_X(G,\varepsilon;A)$ is of rank
$|F|-|F_X|=r(G)-r\langle X\rangle$.
\end{proof}

To see that $|T_{\rm nz}(G,\varepsilon;A)|$ is independent of $\varepsilon$ and
$A$, let $\varrho$ be another orientation of $G$. We introduce an involution
$P_{\varrho,\varepsilon}:A^E\rightarrow A^E$, defined by
\begin{equation}\label{TED}
(P_{\varrho,\varepsilon}f)(x)=[\varrho,\varepsilon](x)f(x), \sp f\in A^E.
\end{equation}
Obviously, $P_{\varepsilon,\varepsilon}$ is the identity map. For orientations
$\varepsilon,\varrho,\rho$,
\begin{equation}\label{TEDTDW}
P_{\rho,\varrho}P_{\varrho,\varepsilon}=P_{\rho,\varepsilon}.
\end{equation}

\begin{lemma}\label{S-Epsilon-Rho}
The involution $P_{\varrho,\varepsilon}$ is a group isomorphism. Moreover,
\begin{gather}
P_{\varrho,\varepsilon}\bigl(T(G,\varepsilon;A)\bigr) = T(G,\varrho;A),\nonumber\\
P_{\varrho,\varepsilon}\bigl(T_{\rm nz}(G,\varepsilon;A)\bigr) = T_{\rm
nz}(G,\varrho;A). \nonumber
\end{gather}
\end{lemma}
\begin{proof} It is obvious that $P_{\varrho,\varepsilon}$ is a group isomorphism.
Let $f\in T(G,\varepsilon;A)$. For any directed circuit
$(C,\varepsilon_{\textsc c})$, we have
\[
\begin{split}
\sum_{x\in C}[\varrho,\varepsilon_{\textsc c}](x)(P_{\varrho,\varepsilon}f)(x)
&=
\sum_{x\in C}[\varrho,\varepsilon_{\textsc c}](x)[\varrho,\varepsilon](x)f(x)\\
&=\sum_{x\in C}[\varepsilon,\varepsilon_{\textsc c}](x)f(x)=0.
\end{split}
\]
Thus $P_{\varrho,\varepsilon}f\in T(G,\varrho;A)$. Similarly, for $g\in
T(G,\varrho;A)$, we have $P_{\varepsilon,\varrho}g\in T(G,\varepsilon;A)$.
Since $P_{\varrho,\varepsilon}$ is an involution, the first identity follows
immediately. The second identity follows from the fact that for $x\in E$,
$(P_{\varrho,\varepsilon}f)(x)\neq 0$ if and only if $f(x)\neq 0$.
\end{proof}

The {\em coloring arrangement} of $G$ with the color set $A$ is the subgroup
arrangement $\mathcal A_{\CL}(G,A)$ of the abelian group $K:=K(G,A)$,
consisting of the subgroups
\begin{gather}
K_e=K_e(G,\varepsilon;A): =\Bigl\{f\in K(G,A)\:\bigl|\: f(u)=f(v)\Bigr\}, \sp
e=uv\in E.
\end{gather}
The set of proper colorings with values in $A$ is the complement
$K(G,A)-\bigcup_{e\in E} K_e$. The {\em tension arrangement} of the
digraph $(G,\varepsilon)$ with the abelian group $A$ is the subgroup
arrangement ${\mathcal A}_{\TN}(G,\varepsilon;A)$ of the abelian
group $T:=T(G,\varepsilon;A)$, consisting of the subgroups
\begin{gather}
T_e=T_e(G,\varepsilon;A): =\Bigl\{f\in T(G,\varepsilon;A)\:\bigl|\:
f(e)=0\Bigr\}, \sp e\in E.
\end{gather}
The set of nowhere-zero tensions of $(G,\varepsilon)$ with values in
$A$ is the complement $T(G,\varepsilon; A)-\bigcup_{e\in E} T_e$.

\begin{thm}\label{Tension-Counting}
Let $A$ be an abelian group, either $|A|=q$ is finite or $A\in\{{\Bbb Z}, {\Bbb
Q}, {\Bbb R},{\Bbb C}\}$. Then the chromatic polynomial $\chi(G,t)$ is given by
\begin{equation}\label{Chromatic-Counting}
\chi(G,q) = \lambda\Biggl(K(G,A)-\bigcup_{e\in E} K_e\Biggr)\Biggl|_{t=q} =
\chi\Bigl(\mathcal A_{\CL}(G,A),q\Bigr),
\end{equation}
and the tension polynomial $\tau(G,t)$ is given by
\begin{equation}\label{Tension-Counting}
\tau(G,q) =\lambda\Biggl(T(G,\varepsilon;A)-\bigcup_{e\in
E}T_e\Biggr)\Biggl|_{t=q} =\chi\Bigl({\mathcal
A}_{\TN}(G,\varepsilon;A),q\Bigr).
\end{equation}
\end{thm}
\begin{proof}
For finite $|A|=q$ or $A=\Bbb Z$, (\ref{Chromatic-Counting}) and
(\ref{Tension-Counting}) are consequences of
Theorem~\ref{Chracteristic-Formula-Thm} by applying the valuation
$\lambda$ to the subgroup arrangements ${\mathcal A}_{\CL}(G,A)$ and
${\mathcal A}_{\TN}(G,\varepsilon;A)$. For $A={\Bbb F}\in\{{\Bbb Q},
{\Bbb R}, {\Bbb C}\}$ to be an infinite field,
(\ref{Chromatic-Counting}) and (\ref{Tension-Counting}) follow from
the valuation $\lambda$ on the hyperplane arrangements ${\mathcal
A}_{\CL}(G,\Bbb F)$ and ${\mathcal A}_{\TN}(G,\varepsilon;\Bbb F)$,
such that $\lambda(W)=t^{\dim(W)}$ for vector spaces $W$ over $\Bbb
F$; see \cite{Chen-Infinite-Subspace-Arrangement,
Ehrenborg-Readdy1}.
\end{proof}

If the abelian group $A$ is a finite field, the counting formula
(\ref{Chromatic-Counting}) was observed by Athanasiadis
\cite{Athanasiadis1} and Bj\"orner and Ekedahl
\cite{Bjorner-Ekedahl1}.

\section{Integral chromatic polynomials}

In this section we consider colorings of a graph $G=(V,E)$ with values in the
field $\Bbb R$ of real numbers. The set of proper colorings is the complement
of the hyperplane arrangement $\mathcal A_{\CL}(G,{\Bbb R})$. Let $q$ be a
positive integer, and let $(0,q)=\bigl\{a\in{\Bbb R}\:|\:0<a<q\bigr\}$,
$[0,q]=\bigl\{a\in{\Bbb R}\:|\:0\leq a\leq q\bigr\}$. Then the set of proper
colorings with the color set $\{1,2,\ldots,q-1\}$ is the set of lattice points
of the dilatation $q\Delta^+_{\CL}(G)$ (dilated by $q$) of the non-convex
polyhedron
\begin{gather}\label{Non-Convex-Cone}
\Delta^+_{\CL}(G): =\big\{f\in{\Bbb R}^V\:|\: f(u)\neq
f(v)\;\mbox{for}\; e=uv\in E,\,0<f<1\big\}.
\end{gather}
For each $e=uv\in E$, the non-equality $f(u)\neq f(v)$ can be split into two
inequalities:
\[
f(u)>f(v)  \sp\mbox{and}\sp f(u)<f(v),
\]
which can be interpreted as two orientations on the edge $e=uv$. Thus, for each
orientation $\varepsilon$ of $G$, we have a system of linear inequalities,
indexed by elements of $E(G)$. The system of such linear inequalities has a
common solution $f\in{\Bbb R}^V$ if and only if the orientation $\varepsilon$
is acyclic.

Recall that $A{\mathcal O}(G)$ is set of all acyclic orientations of $G$. For
each $\varepsilon\in A{\mathcal O}(G)$, we define a 0-1 open polytope
\begin{gather}\label{Convex-Cone}
\Delta^+_\CL(G,\varepsilon):=\big\{f\in{\Bbb R}^V\:|\: f(u)>f(v)
\;\mbox{for}\;u\stackrel{\varepsilon}{\rightarrow}v,\,0<f<1\big\}.
\end{gather}
Such polytopes are in one-to-one correspondence with the connected components
of the complement ${\Bbb R}^V-\bigcup{\mathscr A}_\CL(G,{\Bbb R})$. We call
$\Delta^+_\CL(G,\varepsilon)$ the {\em chromatic open polytope} and its closure
$\bar{\Delta}^+_\CL(G,\varepsilon)$ the {\em chromatic polytope} of the digraph
$(G,\varepsilon)$. Let $\chi(G,\varepsilon;q)$ denote the number of colorings
of the digraph $(G,\varepsilon)$ with the color set $\{1,2,\ldots,q\}$, i.e.,
\begin{equation}
\chi(G,\varepsilon;q): = \#\big\{f:V(G)\rightarrow
\{1,2,\ldots,q\}\:|\:f(u)>f(v)\; \mbox{for}\;
u\stackrel{\varepsilon}{\rightarrow}v\big\}.
\end{equation}
Each coloring $f$ counted in $\chi(G,\varepsilon;q)$ can be viewed as a lattice
point satisfying $0<f<q+1$ and $f(u)>f(v)$ for
$u\stackrel{\varepsilon}{\rightarrow}v$. Then $\chi(G,\varepsilon;q)$ counts
the number of lattice points of the open lattice polytope
$(q+1)\Delta^+_\CL(G,\varepsilon)$, i.e.,
\begin{equation}\label{CPO}
\chi(G,\varepsilon;q) =\bigl|(q+1)\Delta^+_\CL(G,\varepsilon)\cap{\Bbb
Z}^V\bigr|.
\end{equation}
We call $\chi(G,\varepsilon;q)$ the {\em chromatic polynomial of $G$ with
respect to $\varepsilon$}. Let
\begin{equation}
\bar\chi(G,\varepsilon;q): = \#\big\{f:V(G)\rightarrow
\{1,2,\ldots,q\}\:|\:f(u)\geq f(v)\; \mbox{for}\;
u\stackrel{\varepsilon}{\rightarrow}v\big\}.
\end{equation}
Notice that the color set $\{1,2\ldots,q\}$ can be replaced by
$\{0,1,\ldots,q-1\}=[0,q-1]\cap{\Bbb Z}$, and the counting function
$\bar\chi(G,\varepsilon;q)$ remains unchanged. Similarly, a coloring
$f$ counted in $\bar\chi(G,\varepsilon;q)$ with the color set
$\{0,1,\ldots,q-1\}$ can be viewed as a lattice point satisfying
$0\leq f\leq q-1$ and $f(u)\geq f(v)$ for
$u\stackrel{\varepsilon}{\rightarrow}v$. Thus
$\bar\chi(G,\varepsilon;q)$ counts the number of lattice points of
the lattice polytope $(q-1)\bar\Delta^+_\CL(G,\varepsilon)$, i.e.,
\begin{equation}\label{CPC}
\bar\chi(G,\varepsilon;q)=\bigl|(q-1)\bar\Delta^+_\CL(G,\varepsilon)\cap{\Bbb
Z}^V\bigr|.
\end{equation}
Now we have seen that $\chi(G,\varepsilon;q)$ is the Ehrhart polynomial
$L\bigl(\Delta^+_\CL(G,\varepsilon),q+1\bigr)$ of the lattice open polytope
$\Delta_ {\CL}^+(G,\varepsilon)$ at $q+1$, and $\bar\chi(G,\varepsilon;q)$ is
the Ehrhart polynomial $L\bigl(\bar\Delta^+_\CL(G,\varepsilon),q-1\bigr)$ of
the closed polytope $\bar\Delta_ {\CL}^+(G,\varepsilon)$ at $q-1$. \vspace{1ex}

{\sc Proof of Theorem~\ref{CP}.}  Let $\varepsilon\in A{\mathcal O}(G)$. The
polytope $\bar\Delta^+_\CL(G,\varepsilon)$ is the convex hull of the lattice
points $f\in{\Bbb Z}^V$ such that $f(w)=0$ or $1$ for all $w\in V$ and
$f(u)\geq f(v)$ for $u\stackrel{\varepsilon}{\rightarrow}v$. The reciprocity
law is a straightforward consequence of the reciprocity law of Ehrhart
polynomials. In fact, the reciprocity law (\ref{Epsilon-CP-RL}) follows from
(\ref{CPO}) and (\ref{CPC}) as
\begin{eqnarray*}
\chi(G,\varepsilon;-t) &=& L\bigl(\Delta^+_\CL(G,\varepsilon),-t+1\bigr)\\
&=& L\bigl(\Delta^+_\CL(G,\varepsilon),-(t-1)\bigr)\\
&=& (-1)^{|V|}L\bigl(\bar\Delta^+_\CL(G,\varepsilon),t-1\bigr)\\
&=& (-1)^{|V|}\bar\chi(G,\varepsilon;t).
\end{eqnarray*}
In particular, $\bar\chi(G,\varepsilon;1)
=L\bigl(\bar\Delta^+_\CL(G,\varepsilon),0\bigr)=1$ and
\[
\chi(G,\varepsilon;-1) =L\bigl(\Delta^+_\CL(G,\varepsilon),0\bigr) =(-1)^{\dim
\Delta^+_\CL(G,\varepsilon)} =(-1)^{|V|}.
\]

Let $t=q$ be a positive integer. Then (\ref{CP-RL}) follows from the disjoint
union
\begin{eqnarray*}
(q+1)\Delta^+_\CL(G) &=& \bigsqcup_{\varepsilon\in A{\mathcal
O}(G)}(q+1)\Delta^+_\CL(G,\varepsilon),
\end{eqnarray*}
and (\ref{CP-Sum-Closed}) follows from definition of $\bar\chi(G,q)$. The
reciprocity law (\ref{CP-RL}) is a straightforward consequence of
(\ref{CP-Sum-Open})--(\ref{Epsilon-CP-RL}). The interpretation for
$(-1)^{|V|}\chi(G,-1)$ follows from (\ref{CP-Sum-Open}) as
\[
\chi(G,-1) = \sum_{\varepsilon\in A{\mathcal O}(G)} \chi(G,\varepsilon;-1) =
\sum_{\varepsilon\in A{\mathcal O}(G)} (-1)^{|V|} = (-1)^{|V|}\bigl|A{\mathcal
O}(G)\bigr|.
\]
\qed

\begin{rmk}
One may also define integral $q$-colorings $f$ of a graph $G$ as
functions $f:V(G)\rightarrow{\Bbb Z}$ such that $|f(v)|<q$ for all
$v\in V(G)$. Then the counting function $\chi_{\Bbb Z}(G,q)$,
defined as the number of proper integral $q$-colorings of $G$, is a
polynomial function of positive integers $q$ of degree $|V|$.
However, this polynomial $\chi_{\Bbb Z}(G,q)$ is essentially the
same as $\chi(G,q)$, up to a change of variable, i.e.,
\[
\chi_{\Bbb Z}(G,q)=\chi(G,2q-1).
\]
\end{rmk}

\section{Integral tension polynomials}

In this section we consider tensions of digraph $(G,\varepsilon)$
with values in $\Bbb Z$ and $\Bbb R$. Given a positive integer $q$;
let $T(G,\varepsilon;q)$ be the set of all real-valued $q$-tensions
of $(G,\varepsilon)$. We denote by $T_{\Bbb Z}(G,\varepsilon;q)$ the
set of all integral $q$-tensions of $(G,\varepsilon)$, and by
$T_{\rm nz\Bbb Z}(G,\varepsilon;q)$ the set of all nowhere-zero
integral $q$-tensions, i.e.,
\[
T_{\rm nz\Bbb Z}(G,\varepsilon;q): =\big\{f\in T_{\Bbb
Z}(G,\varepsilon;q)\:|\: f(x)\neq 0\;\mbox{for}\; x\in E\big\}.
\]
Clearly, $T_{\rm nz\Bbb Z}(G,\varepsilon;q)$ is the set of lattice points of
the dilatation $q\Delta_{\TN}(G,\varepsilon)$ (dilated by $q$) of the
non-convex polyhedron
\[
\Delta_{\TN}(G,\varepsilon): =\big\{f\in T(G,\varepsilon;{\Bbb
R})\:|\: 0<|f(x)|<1\;\mbox{for}\;x\in E\big\}.
\]
It follows that
\begin{equation}
\tau_{\Bbb Z}(G,q): =\big|T_{\rm nz\Bbb Z}(G,\varepsilon;q)\big|
=L\big(\Delta_{\TN}(G,\varepsilon),q\big)
\end{equation}
is an Ehrhart polynomial function of degree
$\dim\Delta_{\TN}(G,\varepsilon)$ in the positive integral variable
$q$. We shall see that $|T_{\rm nz\Bbb Z}(G,\varepsilon;q)|$ is
independent of the chosen orientation $\varepsilon$; and call
$\tau_{\Bbb Z}(G,q)$ the {\em integral tension polynomial} of $G$.

\begin{lemma}
For orientations $\varepsilon,\varrho\in\mathcal O(G)$ and the involution
$P_{\varrho,\varepsilon}$, we have
\begin{equation}
\begin{split}
P_{\varrho,\varepsilon} \bigl(\Delta_{\TN}(G,\varepsilon)\bigr) &= \Delta_{\TN}(G,\varrho), \nonumber \\
P_{\varrho,\varepsilon} \bigl(T_{\rm nz\Bbb Z}(G,\varepsilon;q)\bigr) &= T_{\rm
nz\Bbb Z}(G,\varrho;q). \nonumber
\end{split}
\end{equation}
\end{lemma}
\begin{proof}
It follows trivially from Lemma~\ref{S-Epsilon-Rho}.
\end{proof}

For nowhere-zero real-valued tensions $f\in T_{\rm
nz}(G,\varepsilon;{\Bbb R})$, the non-equality $f(x)\neq 0$ can be
split into two inequalities:
\[
f(x)>0 \sp\mbox{and}\sp f(x)<0,
\]
which can be interpreted as two orientations on the edge $x$; one is the same
as $\varepsilon(x)$ and the other is opposite to $\varepsilon(x)$. Let
$\varrho$ be an orientation of $G$. We define a convex cone
\[
T(G,\varepsilon;\varrho):=\big\{f\in T(G,\varepsilon;{\Bbb R})\:|\:
[\varrho,\varepsilon](x)f(x)>0\;\mbox{for}\; x\in E\big\}
\]
of the tension vector space $T(G,\varepsilon;{\Bbb R})$. It is clear
that the complement
\[
T(G,\varepsilon;{\Bbb R})-\bigcup_{e\in E(G)} T_e
\]
is a disjoint union of open convex cones $T(G,\varepsilon;\varrho)$,
where $\varrho\in\mathcal O(G)$, but some of them may be empty. Each
such cone $T(G,\varepsilon;\varrho)$ is isomorphic to the open
convex cone
\[
T^+(G,\varrho): =\big\{f\in T(G,\varrho;{\Bbb R})\:|\:
f(x)>0\;\mbox{for}\; x\in E\big\}.
\]
We introduce the open polytope
\[
\Delta_{\TN}(G,\varepsilon;\varrho): =\big\{f\in
T(G,\varepsilon;\varrho)\:|\: |f(x)|<1\;\mbox{for}\;x\in E\big\}
\]
and the 0-1 open polytope
\[
\Delta_{\TN}^+(G,\varepsilon): =\big\{f\in T^+(G,\varepsilon)\:|\:
0<f(x)<1\;\mbox{for}\;x\in E\big\}.
\]
We call $\Delta_{\TN}^+(G,\varepsilon)$ the {\em tension open polytope} and its
closure $\bar\Delta_{\TN}^+(G,\varepsilon)$ the {\em tension polytope} of the
digraph $(G,\varepsilon)$. Let $\tau_{\Bbb Z}(G,\varepsilon;q)$ denote the
number of nowhere-zero tensions with values in $\{0,1,\ldots,q-1\}$. In other
words, $\tau_{\Bbb Z}(G,\varepsilon;q)$ is the number of positive integral
$q$-tensions of $(G,\varepsilon)$. Clearly, $\tau_{\Bbb Z}(G,\varepsilon;q)$
counts the number of lattice points of $q\Delta_{\TN}^+(G,\varepsilon)$, i.e.,
\begin{gather}\label{TGEQ}
\tau_{\Bbb Z}(G,\varepsilon;q):=\bigl|q\Delta_{\TN}^+(G,\varepsilon)\cap{\Bbb
Z}^E\bigr|=L\bigl(\Delta^+_{\TN}(G,\varepsilon),q\bigr).
\end{gather}
We call $\tau_{\Bbb Z}(G,\varepsilon;q)$ the {\em positive tension polynomial
with respect to $\varepsilon$}. Similarly, let $\bar\tau_{\Bbb
Z}(G,\varepsilon;q)$ denote the number of integral tensions of
$(G,\varepsilon)$ with values in $\{0,1,\ldots,q\}$. In other words,
$\bar\tau_{\Bbb Z}(G,\varepsilon;q)$ is the number of nonnegative integral
$(q+1)$-tensions of $(G,\varepsilon)$. Then $\bar\tau_{\Bbb
Z}(G,\varepsilon;q)$ counts the number of lattice points of
$q\bar\Delta_{\TN}^+(G,\varepsilon)$, i.e.,
\begin{gather}\label{BTGEQ}
\bar\tau_{\Bbb
Z}(G,\varepsilon;q):=\bigl|q\bar\Delta_{\TN}^+(G,\varepsilon)\cap{\Bbb
Z}^E\bigr|=L\bigl(\bar\Delta^+_{\TN}(G,\varepsilon),q\bigr).
\end{gather}
We call $\bar\tau_{\Bbb Z}(G,\varepsilon;q)$ the {\em nonnegative tension
polynomial of $G$ with respect to $\varepsilon$}.

\begin{lemma}\label{EDT}
For orientations $\varepsilon,\varrho\in{\mathcal O}(G)$ and the involution
$P_{\varrho,\varepsilon}$,
\begin{gather}
P_{\varrho,\varepsilon}\bigl(\Delta_{\TN}(G,\varepsilon;\varrho)\bigr) =
\Delta_{\TN}^+(G,\varrho),\label{TED-DTN}\\
\Delta_{\TN}(G,\varepsilon) = \bigsqcup_{\varrho\in\mathcal
O(G)}\Delta_{\TN}(G,\varepsilon;\varrho). \label{DTN-DTN}
\end{gather}
\end{lemma}
\begin{proof}
Let $f\in\Delta_{\TN}(G,\varepsilon;\varrho)$. Since
$P_{\varrho,\varepsilon}f\in T(G,\varrho;{\Bbb R})$ and
$P_{\varrho,\varepsilon}f(x)=[\varrho,\varepsilon](x)f(x)>0$ for all $x\in E$,
we have $P_{\varrho,\varepsilon}f\in\Delta_{\TN}^+(G,\varrho)$. Conversely, for
any $g\in\Delta_{\TN}^+(G,\varrho)$, we obviously have
$P_{\varepsilon,\varrho}g\in\Delta_{\TN}(G,\varepsilon;\varrho)$. Hence
(\ref{TED-DTN}) is valid.

The right-hand side of (\ref{DTN-DTN}) is obviously contained in the left-hand
side. Let $f\in\Delta_{\TN}(G,\varepsilon)$, and let $\varrho$ be an
orientation of $G$ such that $\varrho(x)=\varepsilon(x)$ if $f(x)>0$, and
$\varrho(x)\neq\varepsilon(x)$ if $f(x)<0$. Then
$f\in\Delta_{\TN}(G,\varepsilon;\varrho)$. Thus the left-hand side of
(\ref{DTN-DTN}) is contained in its right-hand side.
\end{proof}

The above Lemma~\ref{EDT} shows that for orientations $\varepsilon,\varrho$ and
positive integers $q$,
\[
\tau(G,\varrho;q)=\bigl|q\Delta_{\TN}^+(G,\varrho)\cap{\Bbb Z}^E\bigr|
=\bigl|q\Delta_{\TN}(G,\varepsilon;\varrho)\cap{\Bbb Z}^E\bigr|.
\]
Notice that $\tau_{\Bbb Z}(G,\varrho;q)\equiv 0$ if and only if
$\Delta_{\TN}^+(G,\varrho)=\emptyset$. The following lemma
characterizes $\varrho$ such that $\Delta_{\TN}^+(G,\varrho)$ is
nonempty.

\begin{lemma}\label{Positive-Nonempty}
Let $\varepsilon\in\mathcal O(G)$. Then the digraph $(G,\varepsilon)$ has a
positive real-valued tension if and only if $(G,\varepsilon)$ has no directed
circuit. In other words, $\Delta_{\TN}^+(G,\varepsilon)$ is nonempty if and
only if $\varepsilon$ is acyclic.
\end{lemma}
\begin{proof}
Trivial.
\end{proof}

Theorem~\ref{Integral-TNP} gives an interpretation for the integral
tension polynomial $\tau_{\Bbb Z}(G,t)$ at zero and negative values.
This interpretation is similar to that of Stanley's result on the
chromatic polynomial at negative integers; see \cite{Stanley1}.
\vspace{1ex}

{\sc Proof of Theorem~\ref{Integral-TNP}.} It is clear that
$\Delta_{\TN}^+(G,\varepsilon)$ is a 0-1 open polytope. Since $\tau_{\Bbb
Z}(G,\varepsilon;q)$ and $\bar\tau_{\Bbb Z}(G,\varepsilon;q)$ are Ehrhart
polynomials by (\ref{TGEQ}) and (\ref{BTGEQ}), we automatically have the
reciprocity law (\ref{Epsilon-Integral-TNP-RL}), $\tau_{\Bbb
Z}(G,\varepsilon;0)=(-1)^{r(G)}$, and $\bar\tau_{\Bbb Z}(G,\varepsilon;0)=1$
from the standard properties of Ehrhart polynomials. The degree follows from
Corollary~\ref{TGEGamma}.

By identity (\ref{DTN-DTN}) and Lemma~\ref{Positive-Nonempty}, we
have the disjoint union
\[
\Delta_{\TN}(G,\varepsilon) = \bigsqcup_{\varrho\in A\mathcal O(G)}
P_{\varepsilon,\varrho}\bigl(\Delta_{\TN}^+(G,\varrho)\bigr).
\]
Now the identity (\ref{Integral-TNP-Sum-Open}) follows immediately. The formula
(\ref{Integral-TNP-Sum-Closed}) follows from definition of $\bar\tau_{\Bbb
Z}(G,q)$. The reciprocity law (\ref{TNP-RL}) follows from the reciprocity law
of Ehrhart polynomials. \qed

\section{Interpretation of modular tension polynomial}

In this section we shall give a decomposition for the modular tension
polynomial similar to that of Theorem~\ref{Integral-TNP}. Let $G=(V,E)$ be a
graph with possible loops and multiple edges. Recall that $A{\mathcal O}(G)$ is
the set of all orientations of $G$ without directed circuit. Given a positive
integer $q$; we define $\Mod_q:{\Bbb R}^E\rightarrow({\Bbb R}/q{\Bbb Z})^E$ by
\begin{gather}
(\Mod_qf)(x)=f(x)\,(\mod q), \sp f\in{\Bbb R}^E.
\end{gather}
Then $\Mod_q({\Bbb Z}^E)=({\Bbb Z}/q{\Bbb Z})^E$ is a subgroup of $({\Bbb
R}/q{\Bbb Z})^E$, and $\Mod_q\bigl(T(G,\varepsilon;{\Bbb Z})\bigr)$ is a
subgroup of $\Mod_q({\Bbb Z}^E)$. Let $\varepsilon,\varrho\in\mathcal O(G)$. We
define an involution $Q_{\varrho,\varepsilon}:[0,q]^E\rightarrow[0,q]^E$ by
\begin{equation}
(Q_{\varrho,\varepsilon}f)(x)=\left\{\begin{array}{rl} f(x) & \mbox{if}\sp
\varrho(x)=\varepsilon(x),\\
q-f(x) & \mbox{if}\sp \varrho(x)\neq\varepsilon(x),
\end{array}\right. \sp f\in[0,q]^E.
\end{equation}
For any orientations $\varepsilon,\varrho,\rho\in\mathcal O(G)$ and any
function $f\in[0,q]^E$, it is straightforward to check
\[
(Q_{\rho,\varrho}P_{\varrho,\varepsilon}f)(x)=\left\{
\begin{array}{rl}
f(x) & \mbox{if}\sp \rho(x)=\varrho(x)=\varepsilon(x), \\
-f(x) & \mbox{if}\sp \rho(x)=\varrho(x)\neq \varepsilon(x), \\
q-f(x) & \mbox{if}\sp \rho(x)\neq\varrho(x)=\varepsilon(x), \\
q+f(x) & \mbox{if}\sp \rho(x)\neq\varrho(x)\neq\varepsilon(x),
\end{array}\right.
\]
and
\begin{equation}\label{PQP}
(P_{\varepsilon,\rho} Q_{\rho,\varrho}
P_{\varrho,\varepsilon}f)(x) =\left\{
\begin{array}{ll}
f(x) & \mbox{if}\sp \rho(x)=\varrho(x)=\varepsilon(x), \\
f(x) & \mbox{if}\sp \rho(x)=\varrho(x)\neq \varepsilon(x), \\
f(x)-q & \mbox{if}\sp \rho(x)\neq\varrho(x)=\varepsilon(x),\\
f(x)+q & \mbox{if}\sp \rho(x)\neq\varrho(x)\neq\varepsilon(x).
\end{array}\right.
\end{equation}

\begin{lemma}\label{Surjectivity}
The restrictions $\Mod_q:T_{\Bbb Z}(G,\varepsilon;q)\rightarrow
T(G,\varepsilon;{\Bbb Z}/q{\Bbb Z})$ and $\Mod_q:T_{\rm nz\Bbb
Z}(G,\varepsilon;q)\rightarrow T_{\rm nz}(G,\varepsilon;{\Bbb Z}/q{\Bbb Z})$
are surjective.
\end{lemma}
\begin{proof}
Let $\tilde f\in T(G,\varepsilon;{\Bbb Z}/q{\Bbb Z})$. We construct an integral
$q$-tension $f\in T_{\Bbb Z}(G,\varepsilon;q)$ as follows. We first identify
${\Bbb Z}/q{\Bbb Z}$ as the set $\{0,1,\ldots,q-1\}$. Let $G_0$ be a connected
component of $G$, and let $v_0$ be a fixed vertex of $G_0$. For each vertex $v$
of $G_0$, let $P=v_0v_1\ldots v_n$ be a shortest path from $v_0$ to $v=v_n$ and
$P(v_0,v_i)=v_0v_1\ldots v_i$, $1\leq i\leq n$. Let $\varrho$ be an orientation
of $P$ such that $v_{i-1}\xrightarrow{\varrho}v_i$, $1\leq i\leq n$. Define
\[
f(v_0,v_1)=\tilde f(v_0,v_1)
+\frac{q}{2}\Bigl([\varepsilon,\varrho](v_0,v_1)-1\Bigr),
\]
\begin{align*}
 f(v_{i-1},v_i) &=
\left\{\begin{array}{llc}
\tilde f(v_{i-1},v_i)+q[\varepsilon,\varrho](v_{i-1},v_i) &\mbox{if}& -q<s_i<0, \vspace{.3ex}\\
\tilde f(v_{i-1},v_i)   &\mbox{if}& 0\leq s_i<q, \vspace{.3ex}\\
\tilde f(v_{i-1},v_i)-q[\varepsilon,\varrho](v_{i-1},v_i) &\mbox{if}& q\leq
s_i<2q,
\end{array}\right.
\end{align*}
where
\[
s_i=[\varepsilon,\varrho](v_{i-1},v_i)\tilde f(v_{i-1},v_i)+ \sum_{x\in
P(v_0,v_{i-1})}[\varepsilon,\varrho](x)f(x).
\]
Let $Q(v_0,w_n)$ be another shortest path from $v_0$ to $v=w_n$. It is routine
by the construction of $f$ to verify that
\[
0\leq a_i:=\sum_{x\in P(v_0,v_i)} [\varepsilon,\varrho](x)f(x)<q,
\]
\[
0\leq b_i:= \sum_{x\in Q(v_0,w_i)} [\varepsilon,\varrho](x)f(x)<q,
\]
and $a_i\equiv b_i(\mod q)$ for $0\leq i\leq n$. Let
$g:V\rightarrow[0,q)\cap{\Bbb Z}$ be defined by $g(v_0):=0=a_0$ and
$g(v):=g(v_n)=a_n$. Then $g$ is a well-defined coloring of $G$ with
values in $\{0,1,\ldots,q-1\}$. It follows that the function $f$ is
a well-defined tension of $(G,\varepsilon)$. Since
$[\varepsilon,\varrho](v_{i-1},v_i)f(v_{i-1},v_i)=a_i-a_{i-1}$ for
$1\leq i\leq n$, we see that $|f|<q$. So $f\in T_{\Bbb
Z}(G,\varepsilon;q)$. Finally, it is clear that if $\tilde f$ is
nowhere-zero, then $f$ is nowhere-zero.

A non-constructive proof can be followed from Corollary~\ref{TGEGamma} and a
result of Tutte \cite{Tutte2} on regular abelian groups.
\end{proof}

Recall that an oriented cut in a digraph $(G,\varepsilon)$ is a disjoint union
of directed bonds. However, an oriented cut may be decomposed into a disjoint
union of bonds that may not be necessarily directed. For instance, the edge set
$\{a,b,c,d\}$ of the digraph in {\sc Figure}~\ref{D-Cut-UD-Bond} is an oriented
cut as it can be written as disjoint union of two directed bonds $\{a,b\}$ and
$\{c,d\}$. However, the same cut is the union of the bonds $\{a,d\}$ and
$\{b,c\}$, which are not directed.
\begin{figure}[h]
\centering \subfigure{\psfig{figure=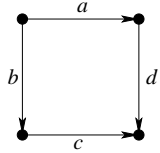,width=20mm}} \caption{An
oriented cut may be decomposed into a disjoint union of undirected bonds.
\label{D-Cut-UD-Bond}}
\end{figure}

\begin{prop}\label{Directed-Cut-Character}
Let $U\subseteq E(G)$ be a nonempty subset of a digraph $(G,\varepsilon)$. Then
$U$ is an oriented cut of $(G,\varepsilon)$ if and only if for any directed
circuit $(C,\varepsilon_{\textsc c})$,
\begin{gather}\label{Cut-Sum-Zero}
\sum_{x\in U\cap C}[\varepsilon,\varepsilon_{\textsc c}](x)=0.
\end{gather}
\end{prop}
\begin{proof} ``$\Rightarrow$": Let $U$ be decomposed into a disjoint union of directed
bonds $B$. Then for any directed circuit $(C,\varepsilon_{\textsc c})$ of $G$,
the intersection $B\cap C$ has even number of edges, half of them agreeing with
the orientation $\varepsilon_{\textsc c}$. It is obvious that
\[
\sum_{x\in B\cap C}[\varepsilon,\varepsilon_{\textsc c}](x)=0.
\]
Hence (\ref{Cut-Sum-Zero}) is valid by adding up these sums for all directed
bonds $B$ in the decomposition of $U$.

``$\Leftarrow$": It is clear that $U$ contains no loops. Choose an edge $e\in
U$ whose ending vertex is $v$. Without loss of generality, we may assume that
the graph $G$ is connected. Let $S$ be a subset of $V$, consisting of vertices
$u$ such that there exists a path $P$ from $u$ to $v$ and
\[
\sum_{x\in U\cap P}[\varepsilon,\varepsilon_{\textsc p}](x)>0,
\]
where $\varepsilon_{\textsc p}$ is the direction of the path $P$ from $u$ to
$v$. Similarly, let $T$ be a subset of $V$, consisting of vertices $w$ such
that there exists a path $Q$ from $w$ to $v$ and
\[
\sum_{y\in U\cap Q}[\varepsilon,\varepsilon_{\textsc q}](x)\leq 0,
\]
where $\varepsilon_{\textsc q}$ is the direction of the path $Q$ from $w$ to
$v$. It is clear that $S$ and $T$ are well defined, and $S\cap T=\emptyset$,
$S\cup T=V$; so $[S,T]$ is a cut. We first claim that $[S,T]$ is contained in
$U$. In fact, if there is an edge $e'=u'v'\in[S,T]$ such that $e'\not\in U$,
then $u'\in S,v'\in T$. Let $P'$ be a path from $u'$ to $v$ and $\sum_{x\in
U\cap P'}[\varepsilon,\varepsilon_{\textsc p'}](x)>0$. Then $Q'=(v',e')P'$ is a
path from $v'$ to $v$ and $\sum_{x\in U\cap
Q'}[\varepsilon,\varepsilon_{\textsc q'}](x)>0$. So $v'\in S$, a contradiction.

Next we claim that $[S,T]$ is a directed cut with all arrows from $S$ to $T$.
Suppose there is an edge $e'$ with initial vertex $v'\in T$ and ending vertex
$u'\in S$. Let $P'$ be a path from $u'$ to $v$ such that $\sum_{x\in U\cap
P'}[\varepsilon,\varepsilon_{\textsc p'}](x)>0$. Then $Q'=(v',e')P'$ is a path
from $v'$ to $v$ and
\[
\sum_{x\in U\cap Q'}[\varepsilon,\varepsilon_{\textsc q'}](x) = 1 +\sum_{x\in
U\cap P'}[\varepsilon,\varepsilon_{\textsc p'}](x)>0.
\]
So $v'\in S$; this is a contradiction.

Now the edge set $[S,T]$ is a directed cut. Let $U'=U-[S,T]$. Then
by induction, $U'$ can be decomposed into a disjoint union of
directed bonds. So is $U$.
\end{proof}

Proposition~\ref{Directed-Cut-Character} can be considered as a dual version of
the degree condition (in-degree equals out-degree) on directed Eulerian graphs.
However, the degree condition (evenness of degrees) on Eulerian graphs has the
following dual version on cuts.

\begin{cor}\label{Cut-Character}
Let $G$ be a graph with possible loops and multiple edges, and let $U\subseteq
E(G)$ be a nonempty subset. Then $U$ is a cut if and only if $|U\cap C|$ is
even for any circuit $C$.
\end{cor}
\begin{proof}
The necessity is obvious. To prove the sufficiency, we still proceed
by induction on $|U|$. For $|U|=1$, to have $|U\cap C|$ to be even
for a circuit $C$, the intersection $U\cap C$ must be empty. So $U$
is a bridge of $G$. Of course it is a cut. Let $|U|\geq 2$. Choose
an edge $e\in U$ and all circuits $C_1,\ldots,C_n$ that contain $e$.
For each $C_i$, choose an edge $e_i\in U\cap(C_i-\{e\})$ as $|U\cap
C_i|$ is even. It is clear that $B=e\cup\{e_i\:|\:1\leq i\leq n\}$
is a bond of $G$. Now for any circuit $C$, we have $|B\cap C|=even$
and $|U\cap C|=even$. It follows that $|(U-B)\cap C|=even$. If
$U-B=\emptyset$, then $U=B$ is a cut. If $U-B\neq\emptyset$, then by
induction $U-B$ is a cut. Hence $U=(U-B)\cup B$ is a cut.
\end{proof}

\begin{rmk} Proposition~\ref{Directed-Cut-Character} may be
known, but there is no good reference on the statement.
Corollary~\ref{Cut-Character} is well known, its inclusion here is just for
completeness. Directed circuit $(C,\varepsilon_{\textsc c})$ in
Proposition~\ref{Directed-Cut-Character} and circuit $C$ in
Corollary~\ref{Cut-Character} can be replaced by directed Eulerian subgraph and
Eulerian graph respectively.
\end{rmk}

\begin{lemma}\label{Cut-Equiv-Property}
\begin{enumerate}
\item[(a)] The relation $\sim$ is an equivalence relation on $\mathcal O(G)$.

\item[(b)] Let $\varepsilon,\varrho\in\mathcal O(G)$ be cut-equivalent. If
$\varepsilon$ is acyclic, so is $\varrho$.

\item[(c)] Let $\varepsilon,\varrho\in\mathcal O(G)$ be cut-equivalent. Then
$Q_{\varrho,\varepsilon}: q\bar\Delta^+_{\TN}(G,\varepsilon)\rightarrow
q\bar\Delta^+_{\TN}(G,\varrho)$ is a bijection, sending lattice points to
lattice points. In particular,
\begin{gather}
Q_{\varrho,\varepsilon}\bigl(q\Delta^+_{\TN}(G,\varepsilon)\bigr)=q\Delta^+_{\TN}(G,\varrho), \nonumber\\
\tau_{\Bbb Z}(G,\varepsilon;q)=\tau_{\Bbb Z}(G,\varrho;q), \nonumber\\
\bar\tau_{\Bbb Z}(G,\varepsilon;q)=\bar\tau_{\Bbb Z}(G,\varrho;q). \nonumber
\end{gather}
\end{enumerate}
\end{lemma}
\begin{proof}
(a) The reflexivity and symmetry property are obvious. Let
$\varepsilon_i\in\mathcal O(G)$ ($i=1,2,3$) be such that
$\varepsilon_1\sim\varepsilon_2$ and
$\varepsilon_2\sim\varepsilon_3$. Let $\{P_1,T_1\}$ and
$\{P_2,T_2\}$ be partitions of the vertex set $V$ such that
\[
[P_1,T_1]=E(\varepsilon_1\neq\varepsilon_2), \sp
[P_2,T_2]=E(\varepsilon_2\neq\varepsilon_3).
\]
If $(P_1\cap P_2)\cup(T_1\cap T_2)=\emptyset$, we have $P_1=T_2$, $P_2=T_1$; if
$(P_1\cap T_2)\cup(P_2\cap T_1)=\emptyset$, we have $P_1=P_2$, $T_1=T_2$. In
either case the two cuts are the same. Hence $\varepsilon_1=\varepsilon_3$. If
it is neither of the two cases, let $S=(P_1\cap P_2)\cup(T_1\cap T_2)$,
$T=(P_1\cap T_2)\cup(P_2\cap T_1)$. Then $[S,T]$ is a cut and can be written as
a disjoint union of sets
\[
[P_1\cap P_2,P_1\cap T_2],\; [P_1\cap P_2,P_2\cap T_1],\; [T_1\cap T_2,P_1\cap
T_2],\; [T_1\cap T_2,P_2\cap T_1].
\]
It is routine to verify that $[S,T]=E(\varepsilon_1\neq\varepsilon_3)$.

Next we show that the cut $[S,T]$ is an oriented cut. Let
$(C,\varepsilon_{\textsc c})$ be a directed circuit. Then
\[
\begin{split}
\sum_{x\in C\cap E(\varepsilon_1\neq\varepsilon_3)}
[\varepsilon_1,\varepsilon_{\textsc c}](x) & = \Biggl\{\sum_{x\in C\cap
E(\varepsilon_1=\varepsilon_2\neq\varepsilon_3)} + \sum_{x\in C\cap
E(\varepsilon_1\neq \varepsilon_2=\varepsilon_3)}\Biggr\}
[\varepsilon_1,\varepsilon_{\textsc c}](x)\\
&= \Biggl\{\sum_{x\in C\cap E(\varepsilon_2\neq\varepsilon_3)} - \sum_{x\in
C\cap E(\varepsilon_1\neq \varepsilon_2\neq\varepsilon_3)}\Biggr\}
[\varepsilon_2,\varepsilon_{\textsc c}](x)\\
&\quad - \Biggl\{\sum_{x\in C\cap E(\varepsilon_1\neq\varepsilon_2)} -
\sum_{x\in C\cap E(\varepsilon_1\neq \varepsilon_2\neq\varepsilon_3)}\Biggr\}
[\varepsilon_2,\varepsilon_{\textsc c}](x) \\
&= \Biggl\{\sum_{x\in C\cap E(\varepsilon_2\neq\varepsilon_3)} - \sum_{x\in
C\cap E(\varepsilon_1\neq \varepsilon_2)}\Biggr\}
[\varepsilon_2,\varepsilon_{\textsc c}](x)=0.
\end{split}
\]
The last equality follows from Proposition~\ref{Directed-Cut-Character}.

(b) Suppose $\varrho$ is not acyclic, i.e., there is a directed
circuit $(C,\varrho_{\textsc c})$ of the digraph $(G,\varrho)$.
Since $(G,\varepsilon)$ has no directed circuit, the circuit $C$
must intersect $E(\varepsilon\neq\varrho)$. Thus any bond of
$E(\varepsilon\neq\varrho)$ that intersects $(C,\varrho_{\textsc
c})$ must have opposite directed edges. This is a contrary to that
$E(\varepsilon\neq\varrho)$ is an oriented cut.

(c) Let $C$ be a circuit of $G$ with an orientation $\varepsilon_{\textsc c}$.
For any $f\in q\bar\Delta^+_{\TN}(G,\varepsilon)$, we have
\[
\sum_{x\in C}[\varepsilon,\varepsilon_{\textsc c}](x)f(x)=0.
\]
Notice that $\varepsilon$ and $\varrho$ have the same direction at
edges of $C\cap E(\varepsilon=\varrho)$ and have the opposite
directions at edges of $C\cap E(\varepsilon\neq\varrho)$. Then
\[
\begin{split}
\sum_{x\in C}[\varrho,\varepsilon_{\textsc c}](x)(Q_{\varrho,\varepsilon}f)(x)
& =\sum_{x\in C\cap E(\varepsilon=\varrho)}[\varepsilon,\varepsilon_{\textsc
c}](x)f(x)
\nonumber \\
& \quad - \sum_{x\in C\cap
E(\varepsilon\neq\varrho)}[\varepsilon,\varepsilon_{\textsc c}](x)\bigl(q-f(x)\bigr) \nonumber \\
& = \sum_{x\in C}[\varepsilon,\varepsilon_{\textsc c}](x)f(x)-q\sum_{x\in C\cap
E(\varepsilon\neq\varrho)}[\varepsilon,\varepsilon_{\textsc c}](x)=0.
\end{split}
\]
This means that $Q_{\varrho,\varepsilon}f$ is a nonnegative
$q$-tension of $(G,\varrho)$. The other identities can be verified
trivially.
\end{proof}

Part (b) of Lemma~\ref{Cut-Equiv-Property} shows that the cut-equivalence
relation $\sim$ can be viewed as an equivalence relation on the set $A\mathcal
O(G)$ of acyclic orientations. We denote by $[A\mathcal O(G)]$ the quotient set
$A\mathcal O(G)/\!\!\sim$ of cut-equivalence classes. For $\varepsilon\in
A\mathcal O(G)$, we denote by $[\varepsilon]$ the equivalence class in
$A\mathcal O(G)/\!\!\sim$, and define
\begin{gather}
\tau_{\Bbb Z}\bigl(G,[\varepsilon];q\bigr):=\tau_{\Bbb Z}(G,\varepsilon;q), \sp
\bar\tau_{\Bbb Z}\bigl(G,[\varepsilon];q\bigr):=\bar\tau_{\Bbb
Z}(G,\varepsilon;q).
\end{gather}
Let $\bar\tau(G,q)$ denote the number of pairs $(f,[\varepsilon])$,
where $[\varepsilon]$ is a cut-equivalence class of $\varepsilon$ in
$[A\mathcal O(G)]$ and $f$ is an integral $(q+1)$-tension of
$(G,\varepsilon)$, i.e.,
\begin{align}
\bar\tau(G,q):&=\#\big\{(f,\varepsilon)\:|\:
\varepsilon\in[A\mathcal
O(G)],\;\mbox{$f$ is an integral tension} \nonumber\\
&\hspace{11mm} \mbox{of $(G,\varepsilon)$ such that $0\leq f(x)\leq
q$ for $x\in E$}\big\}.
\end{align}
Let $f$ be a real-valued tension of $(G,\varepsilon)$. We define an orientation
$\varepsilon_f$ on $G$ as
\begin{equation}\label{EFX}
\varepsilon_f(x)=\left\{\begin{array}{rl} \varepsilon(x) & \mbox{if}\; f(x)>0,
\\ -\varepsilon(x) & \mbox{if}\; f(x)\leq 0,
\end{array}\right. \sp x\in E.
\end{equation}
Let $\varrho$ be an orientation of $G$. We define a 0-1 function
$f_\varrho:E\rightarrow\{0,1\}$ as
\begin{equation}\label{FRX}
f_\varrho(x)=\left\{\begin{array}{ll} 1 & \mbox{if}\;
\varrho(x)=\varepsilon(x), \\
0  & \mbox{if}\; \varrho(x)\neq\varepsilon(x),
\end{array}\right. \sp x\in E.
\end{equation}

\begin{lemma}\label{Induced-Orientation-C-equivalent}
Let $f_i\in T_{\rm nz\Bbb Z}(G,\varepsilon;q)$, and let
$\varepsilon_i=\varepsilon_{f_i}$ be orientations defined by (\ref{EFX}). If
$f_1(x)\equiv f_2(x)\,(\mod q)$ for all $x\in E(G)$, then $\varepsilon_1$ and
$\varepsilon_2$ are cut-equivalent.
\end{lemma}
\begin{proof}
Since $f_1(x)\equiv f_2(x)(\mod q)$ for all $x\in E(G)$, we have
\[
f_1(x)=f_2(x)+k(x)q, \sp x\in E,
\]
where $k(x)\in\{-1,0\}$ if $f_2(x)>0$, and $k(x)\in\{0,1\}$ if $f_2(x)<0$. More
precisely,
\[
f_1(x)=\left\{\begin{array}{lll} f_2(x) & \mbox{if} &
\varepsilon_1(x)=\varepsilon_2(x), \\
f_2(x)-q & \mbox{if} &
\varepsilon_1(x)\neq \varepsilon_2(x)=\varepsilon(x),\\
f_2(x)+q & \mbox{if} & \varepsilon_1(x)\neq \varepsilon_2(x)\neq\varepsilon(x).
\end{array}\right.
\]
Let $g_i=P_{\varepsilon_i,\varepsilon}f_i$. Then $g_i$ are positive
$q$-tensions of $(G,\varepsilon_i)$. Moreover,
\begin{equation}\label{QgI}
g_1(x)=\left\{\begin{array}{rll} g_2(x) & \mbox{if} &
\varepsilon_1(x)=\varepsilon_2(x), \\
q-g_2(x) & \mbox{if} & \varepsilon_1(x)\neq \varepsilon_2(x).
\end{array}\right.
\end{equation}

Let $C$ be a circuit of $G$ with a direction $\varepsilon_{\textsc c}$. Since
$g_i$ are positive tensions of $(G,\varepsilon_i)$, then
\[
\sum_{x\in C}[\varepsilon_1,\varepsilon_{\textsc c}](x)g_1(x)=\sum_{x\in
C}[\varepsilon_2,\varepsilon_{\textsc c}](x)g_2(x)=0.
\]
Applying (\ref{QgI}) and cancelling the common terms on both sides, we have
\begin{align}
\sum_{x\in C\cap
E(\varepsilon_1\neq\varepsilon_2)}[\varepsilon_1,\varepsilon_{\textsc
c}](x)g_1(x) &= \sum_{x\in C\cap E(\varepsilon_1\neq\varepsilon_2)}
[\varepsilon_2,\varepsilon_{\textsc c}](x)g_2(x)
\nonumber\\
&= -\sum_{x\in C\cap E(\varepsilon_1\neq\varepsilon_2)}
[\varepsilon_1,\varepsilon_{\textsc c}](x)\bigl(q-g_1(x)\bigr). \nonumber
\end{align}
It follows that
\[
\sum_{x\in C\cap
E(\varepsilon_1\neq\varepsilon_2)}[\varepsilon_1,\varepsilon_{\textsc c}](x)=0.
\]
By Proposition~\ref{Directed-Cut-Character},
$E(\varepsilon_1\neq\varepsilon_2)$ is an oriented cut. Hence
$\varepsilon_1\sim\varepsilon_2$.
\end{proof}

\begin{lemma}\label{Minverse}
Let $\varrho,\rho,\sigma\in A\mathcal O(G)$ be such that
$\rho\sim\sigma\sim\varrho$, and $f\in q\Delta_{\TN}(G,\varepsilon;\varrho)$.
Then
\begin{enumerate}
\item[(a)] $\varepsilon_f=\varrho$.

\item[(b)] $P_{\varepsilon,\rho} Q_{\rho,\varrho} P_{\varrho,\varepsilon}
\big(q\Delta_{\TN}(G,\varepsilon;\varrho)\big)
=q\Delta_{\TN}(G,\varepsilon;\rho)$.

\item[(c)] $P_{\varepsilon,\rho} Q_{\rho,\varrho}
P_{\varrho,\varepsilon}f=P_{\varepsilon,\sigma} Q_{\sigma,\varrho}
P_{\varrho,\varepsilon}f$ if and only if $\rho=\sigma$.

\item[(d)] $T(G,\varepsilon;q)\cap \Mod_q^{-1}\bigl(\Mod_q f\bigr) =
\bigl\{P_{\varepsilon,\alpha} Q_{\alpha,\varrho} P_{\varrho,\varepsilon}f \:|\:
\alpha\sim\varrho\bigr\}$.
\end{enumerate}
\end{lemma}
\begin{proof}
(a) Notice that $\varepsilon_f=\frac{f}{|f|}\varepsilon$ by
\eqref{EFX}, and $[\varrho,\varepsilon](x)f(x)>0$ for all $x\in
E(G)$. It then follows that $\varepsilon_f=\varrho$.

(b) Recall Lemma~\ref{EDT} and Lemma~\ref{Cut-Equiv-Property}(c); we have
\begin{gather}
P_{\varrho,\varepsilon} \bigl(q\Delta_{\TN}(G,\varepsilon;\varrho)\bigr)
=q\Delta^+_{\TN}(G,\varrho), \sp
Q_{\rho,\varrho}\bigl(q\Delta^+_{\TN}(G,\varrho)\bigr)
=q\Delta^+_{\TN}(G,\rho), \nonumber
\end{gather}
and $P_{\varepsilon,\rho} \big(q\Delta^+_{\TN}(G,\rho)\big)
=q\Delta^+_{\TN}(G,\varepsilon;\rho)$. The required identity follows
immediately.

(c) Let $g=P_{\varepsilon,\rho} Q_{\rho,\varrho} P_{\varrho,\varepsilon}f$ and
$h=P_{\varepsilon,\sigma} Q_{\sigma,\varrho} P_{\varrho,\varepsilon}f$. Then
$g\in q\Delta_{\TN}(G,\varepsilon;\rho)$ and $h\in
q\Delta_{\TN}(G,\varepsilon;\sigma)$ by (b). Thus
$\rho=\varepsilon_g=\varepsilon_h=\sigma$ by (a).

(d) Let $\alpha\in\mathcal O(G)$ be such that $\alpha\sim\varrho$. Then by
(\ref{PQP}),
\[
(P_{\varepsilon,\alpha} Q_{\alpha,\varrho}
P_{\varrho,\varepsilon}f)(x)=f(x)+k(x)q, \sp\mbox{where}\sp x\in E(G),\;
k(x)\in{\Bbb Z}.
\]
Clearly, $(P_{\varepsilon,\alpha} Q_{\alpha,\varrho}
P_{\varrho,\varepsilon}f)(x)\equiv f(x)\,(\mod q)$. Thus
$P_{\varepsilon,\alpha} Q_{\alpha,\varrho}
P_{\varrho,\varepsilon}f\in\Mod_q^{-1}\bigl(\Mod_q f\bigr)$.

Let $g\in T(G,\varepsilon;q)$ be such that $\Mod_q g=\Mod_q f$. Then
\begin{equation}\label{FGH}
g(x)=f(x)+h(x)q, \sp\mbox{where}\sp x\in E(G),\; h(x)\in\{-1,0,1\}.
\end{equation}
Since $f$ is nowhere-zero, by (a) and (\ref{FRX}) we have
$\varrho=\varepsilon_f=\frac{f}{|f|}\varepsilon$. Notice that
\begin{enumerate}
\item $\varepsilon_g(x)=\varepsilon_f(x)=\varepsilon(x)$ $\Leftrightarrow$
$f(x)>0$, $g(x)>0$;

\item $\varepsilon_g(x)=\varepsilon_f(x)\neq \varepsilon(x)$ $\Leftrightarrow$
$f(x)<0$, $g(x)\leq 0$;

\item $\varepsilon_g(x)\neq\varepsilon_f(x)=\varepsilon(x)$ $\Leftrightarrow$
$f(x)>0$, $g(x)\leq 0$;

\item $\varepsilon_g(x)\neq\varepsilon_f(x)\neq\varepsilon(x)$
$\Leftrightarrow$ $f(x)<0$, $g(x)>0$.
\end{enumerate}
The equation \eqref{FGH} implies that the function $g$ must have the form
\[
g(x) =\left\{
\begin{array}{ll}
f(x) & \mbox{if}\sp \varepsilon_g(x)=\varepsilon_f(x)=\varepsilon(x), \\
f(x) & \mbox{if}\sp \varepsilon_g(x)=\varepsilon_f(x)\neq \varepsilon(x), \\
f(x)-q & \mbox{if}\sp \varepsilon_g(x)\neq\varepsilon_f(x)=\varepsilon(x),\\
f(x)+q & \mbox{if}\sp \varepsilon_g(x)\neq\varepsilon_f(x)\neq\varepsilon(x).
\end{array}\right.
\]
On the other hand, recall the formula \eqref{PQP}; we have
\[
(P_{\varepsilon,\varepsilon_g} Q_{\varepsilon_g,\varepsilon_f}
P_{\varepsilon_f,\varepsilon}f)(x) =\left\{
\begin{array}{ll}
f(x) & \mbox{if}\sp \varepsilon_g(x)=\varepsilon_f(x)=\varepsilon(x), \\
f(x) & \mbox{if}\sp \varepsilon_g(x)=\varepsilon_f(x)\neq \varepsilon(x), \\
f(x)-q & \mbox{if}\sp \varepsilon_g(x)\neq\varepsilon_f(x)=\varepsilon(x),\\
f(x)+q & \mbox{if}\sp \varepsilon_g(x)\neq\varepsilon_f(x)\neq\varepsilon(x).
\end{array}\right.
\]
Therefore $g =P_{\varepsilon,\varepsilon_g} Q_{\varepsilon_g,\varepsilon_f}
P_{\varepsilon_f,\varepsilon}f =P_{\varepsilon,\varepsilon_g}
Q_{\varepsilon_g,\varrho} P_{\varrho,\varepsilon}f$.
\end{proof}

\begin{prop}\label{0-1-cut-equiv}
Fix an orientation $\varepsilon\in\mathcal O(G)$. The number of orientations of
$G$ that are cut-equivalent to $\varepsilon$ is the number of 0-1 tensions of
the digraph $(G,\varepsilon)$, i.e.,
\begin{gather}
\bar\tau_{\Bbb
Z}(G,\varepsilon;1)=\bigl|\bar\Delta_{\TN}^+(G,\varepsilon)\cap{\Bbb
Z}^E\bigr|.
\end{gather}
\end{prop}
\begin{proof}
Let $\varrho$ be an orientation which is cut-equivalent to $\varepsilon$. We
identity $\varrho$ as a 0-1 function $f_{-\varrho}$ given by (\ref{FRX}). We
claim that $f_{-\varrho}$ is a 0-1 tension of $(G,\varepsilon)$. In fact, for
any circuit $C$ with an orientation $\varepsilon_{\textsc c}$,
\[
\sum_{x\in C}[\varepsilon,\varepsilon_{\textsc c}](x)f_{-\varrho}(x)
=\sum_{x\in C\cap E(\varrho\neq\varepsilon)} [\varepsilon,\varepsilon_{\textsc
c}](x)=0.
\]
The last equality follows from
Proposition~\ref{Directed-Cut-Character}. So
$f_{-\varrho}\in\bar\Delta^+_{\TN}(G,\varepsilon)\cap{\Bbb Z}^E$.

Conversely, for any 0-1 tension $f$ of $(G,\varepsilon)$, we identify $f$ as an
orientation $-\varepsilon_f$ defined by (\ref{EFX}). We claim that
$-\varepsilon_f\sim\varepsilon$. Since $E(-\varepsilon_f\neq\varepsilon)
=\bigl\{x\in E\:|\:f(x)=1\bigr\}$, then for any circuit $C$ with an orientation
$\varepsilon_{\textsc c}$,
\[
\sum_{x\in C\cap E(-\varepsilon_f\neq\varepsilon)}
[\varepsilon,\varepsilon_{\textsc c}](x) =\sum_{x\in C}
[\varepsilon,\varepsilon_{\textsc c}](x)f(x) =0.
\]
It follows from Theorem~\ref{Directed-Cut-Character} that
$E(-\varepsilon_f\neq\varepsilon)$ is an oriented cut. Hence
$-\varepsilon_f\sim\varepsilon$.

Now it is easy to verify $f=f_{\varepsilon_f}$ and
$\varrho=-\varepsilon_{f_{-\varrho}}=\varepsilon_{f_\varrho}$. So
$\varrho\mapsto f_{-\varrho}$ is a one-to-one correspondence from
the set $[\varepsilon]$ of orientations cut-equivalent to
$\varepsilon$ to the set $\bar\Delta_{\TN}^+(G,\varepsilon)\cap{\Bbb
Z}^E$ of 0-1 tensions of $(G,\varepsilon)$.
\end{proof}

{\sc Proof of Theorem~\ref{Modular-TNP-Decom-RL}.}
We have seen that $\tau(G,t)$ is a polynomial of degree $r(G)$. Fix an
orientation $\varrho\in A\mathcal O(G)$. For any
$f\in\Delta_{\TN}(G,\varepsilon;\varrho)$, by Lemma~\ref{Minverse} and
Proposition~\ref{0-1-cut-equiv}, we have
\begin{align}
\big|T(G,\varepsilon;q)\cap \Mod_q^{-1}\big(\Mod_qf\bigr)\big| & =
\#\big\{ P_{\varepsilon,\rho}
Q_{\rho,\varrho} P_{\varrho,\varepsilon}f \:|\: \rho\sim\varrho\big\} \nonumber\\
& =\#\big\{\rho\in A\mathcal O(G)\:|\:\rho\sim\varrho\big\} \nonumber\\
& =\bar\tau_{\Bbb Z}(G,\varrho;1) \label{Fiber}
\end{align}
We claim that
\begin{equation}\label{SRMM}
\bigsqcup_{\rho\sim\varrho} q\Delta_{\TN}(G,\varepsilon;\rho)
=T(G,\varepsilon;q)\cap \Mod_q^{-1}\Mod_q\Biggl(\bigsqcup_{\rho\sim\varrho}
q\Delta_{\TN}(G,\varepsilon;\rho)\Biggr).
\end{equation}
Obviously, the left side of (\ref{SRMM}) is contained in its right side. For
any $g\in\bigcup_{\rho\sim\varrho} q\Delta_{\TN}(G,\varepsilon;\rho)$ and any
$h\in T(G,\varepsilon;q)\cap\Mod_q^{-1}(\Mod_qg)$, by Lemma~\ref{Minverse},
there exists a $\sigma\in\mathcal O(G)$ such that $\sigma\sim\rho$ and
$h=P_{\varepsilon,\sigma} Q_{\sigma,\rho} P_{\rho,\varepsilon}g\in
q\Delta_{\TN}(G,\varepsilon;\sigma)$. As $\rho\sim\varrho$ and $\sim$ is an
equivalence relation, we have $\sigma\sim\varrho$. This means that $h$ belongs
to the left-hand side of (\ref{SRMM}). Moreover, applying (\ref{Fiber}) and
Lemma~\ref{Cut-Equiv-Property}(c), we see that
\[
\bigl|T(G,\varepsilon;q)\cap
\Mod_q^{-1}\bigl(\Mod_qg\bigr)\bigr|=\bar\tau_{\Bbb Z}(G,\rho;1)
=\bar\tau_{\Bbb Z}(G,\varrho;1).
\]
Apply (\ref{TED-DTN}) of Lemma~\ref{EDT} and Lemma~\ref{Cut-Equiv-Property}(c)
to each set $q\Delta_{\TN}(G,\varepsilon;\rho)$ in the left-hand side of
(\ref{SRMM}); we see that
\[
\bigl|q\Delta_{\TN}(G,\varepsilon;\rho)\bigr|
=\bigl|q\Delta^+_{\TN}(G,\rho)\bigr| =\tau_{\Bbb Z}(G,\rho;q) =\tau_{\Bbb
Z}(G,\varrho;q).
\]
Notice that the left-hand side of (\ref{SRMM}) is a disjoint union
of $\bar\tau_{\Bbb Z}(G,\varrho;1)$ sets. Thus
\[
\tau_{\Bbb Z}(G,\varrho;q)\, \bar\tau_{\Bbb Z}(G,\varrho;1) =
\Bigl|\Mod_q\Bigl(\bigsqcup_{\rho\sim\varrho}
q\Delta_{\TN}(G,\varepsilon;\rho)\Bigr)\Bigr|\, \bar\tau_{\Bbb Z}(G,\varrho;1).
\]
Since $\bar\tau_{\Bbb Z}(G,\varrho;1)\geq 1$, it follows that
\[
\tau_{\Bbb Z}(G,\varrho;q) =\Bigl|\Mod_q\Bigl(\bigsqcup_{\rho\sim\varrho}
q\Delta_{\TN}(G,\varepsilon;\rho)\Bigr)\Bigr|.
\]

Now applying \eqref{DTN-DTN} of Lemma~\ref{EDT} and
Lemma~\ref{Positive-Nonempty}, we have
\[
\begin{split}
T_{\rm nz{\Bbb Z}}(G,\varepsilon;q) &= q\Delta_{\TN}(G,\varepsilon)\cap{\Bbb
Z}^E \\
&= \bigsqcup_{\varrho\in A\mathcal O(G)}
q\Delta_{\TN}(G,\varepsilon;\varrho)\cap{\Bbb Z}^E \\
&= \bigsqcup_{\varrho\in [A\mathcal O(G)]} \bigsqcup_{\rho\sim\varrho}
q\Delta_{\TN}(G,\varepsilon;\rho)\cap{\Bbb Z}^E.
\end{split}
\]
Since \eqref{SRMM} and $\Mod_q$ is surjective (Lemma~\ref{Surjectivity}), we
obtain the following disjoint decomposition
\begin{eqnarray}
T_{\rm nz}(G,\varepsilon;{\Bbb Z}/q{\Bbb Z}) &=& \Mod_q\bigl(T_{\rm nz{\Bbb
Z}}(G,\varepsilon;q)\bigr) \nonumber\\
&=& \bigsqcup_{\varrho\in[A\mathcal O(G)]}
\Mod_q\Bigl(\bigsqcup_{\rho\sim\varrho}
q\Delta_{\TN}(G,\varepsilon;\rho)\cap{\Bbb Z}^E\Bigr). \label{TNZGEZQZ}
\end{eqnarray}
Counting both sides of (\ref{TNZGEZQZ}), the identity (\ref{TGT}) is obtained
as
\[
\tau(G,q) = \sum_{\varrho\in[A\mathcal O(G)]} \tau_{\Bbb Z}(G,\varrho;q).
\]
The identity (\ref{BTGT}) follows from its definition. The
reciprocity law follows from the reciprocity law of the Ehrhart
polynomials $\tau_{\Bbb Z}(G,\varepsilon;q)$ and $\bar\tau_{\Bbb
Z}(G,\varepsilon;q)$. Since $\bar\tau_{\Bbb Z}(G,\varepsilon;0)=1$
for all $\varepsilon\in[A\mathcal O(G)]$, we see that
\[
\tau(G,0)=(-1)^{r(G)}\bar\tau(G,0)=(-1)^{r(G)}\big|[A\mathcal(G)]\big|.
\]
\qed.

\section{Connection with the Tutte polynomial}

The Tutte polynomial (see \cite{Bollobas1}) of a graph $G=(V,E)$ is a
polynomial in two variables
\begin{gather}
T(G;x,y)=\sum_{A\subseteq E(G)}(x-1)^{r\langle E\rangle-r\langle
A\rangle}(y-1)^{n\langle A\rangle},
\end{gather}
where $\langle A\rangle=(V,A)$, $r(G)=|V|-k(G)$, $k(G)$ is the
number of connected components of $G$, $r\langle
A\rangle=|V|-k(\langle A\rangle)$, $k(\langle A\rangle)$ is the
number of connected components of $\langle A\rangle$, and $n\langle
A\rangle=|A|-r\langle A\rangle$. The polynomial $T(G;x,y)$ satisfies
the Deletion-Contraction Relation
\[
T(G;x,y)=\left\{\begin{array}{ll} xT(G/e;x,y) &\mbox{if $e$ is a bridge}, \\
yT(G-e;x,y) & \mbox{if $e$ is a loop}, \\
T(G-e;x,y)+T(G/e;x,y) & \mbox{otherwise}.
\end{array}\right.
\]
It is well-known that the chromatic polynomial $\chi(G,t)$ is related to
$T(G;x,y)$ by
\begin{gather}
\chi(G,t)=(-1)^{r(G)}t^{k(G)}T(G;1-t,0).
\end{gather}
Since $\chi(G,t)=t^{k(G)}\tau(G,t)$, it follows that
\[
\tau(G,t)=(-1)^{r(G)}T(G;1-t,0).
\]
Thus
\[
\bar\tau(G,t)=(-1)^{r(G)}\tau(G,-t)=T(G;t+1,0).
\]
We conclude the information as the following proposition.

\begin{prop}
Let $G=(V,E)$ be a graph with possible loops and multiple edges. Then
\begin{align}
T(G;t,0) &= \bar\tau(G,t-1)=(-1)^{r(G)}\tau(G,1-t).
\end{align}
In particular, $T(G;1,0)=(-1)^{r(G)}\tau(G,0)=\bar\tau(G,0)$ counts the number
of cut-equivalence classes of acyclic orientations of $G$.
\end{prop}

\begin{exmp}
Let $G$ be the labelled graph in {\sc Figure}~\ref{Double-triangle} with $4$
vertices and 5 edges.
\begin{figure}[h]
\centering \subfigure{\psfig{figure=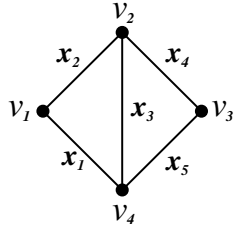,width=30mm}} \caption{A
labelled double triangle graph.} \label{Double-triangle}
\end{figure}
Its chromatic polynomial and modular tension polynomial are given by
\[
\chi(G,t)=t(t-1)(t-2)^2, \sp \tau(G,t)=(t-1)(t-2)^2.
\]
Take the directions $v_1v_2v_4v_1$ and $v_2v_3v_4v_2$ for the left and right
circuits in {\sc Figure}~\ref{Double-triangle}; we see that the number of
nowhere-zero integral $q$-tensions of $G$ is the number of integral solutions
of the system of linear equalities and inequalities
\[
x_1+x_2+x_3=0, \sp -x_3+x_4+x_5=0,\sp -q<x_i<q,\sp x_i\neq 0.
\]
Counting the number of integral solutions of the system gives the integral
tension polynomial
\[
\tau_{\Bbb Z}(G,q)=6(q-1)^2(2q-3)+\frac{(q-1)q(2q-1)}{3}.
\]
Then $|\chi(G,-1)|=|\tau(G,-1)|=|\tau_{\Bbb Z}(G,0)|=18$ is the number of
acyclic orientations of $G$, and $|\tau(G,0)|=4$ is the number of
cut-equivalence classes of acyclic orientations; see {\sc
Figure}~\ref{Double-tiangle-acyclic} below.
\begin{figure}[h]
\centering \subfigure{\psfig{figure=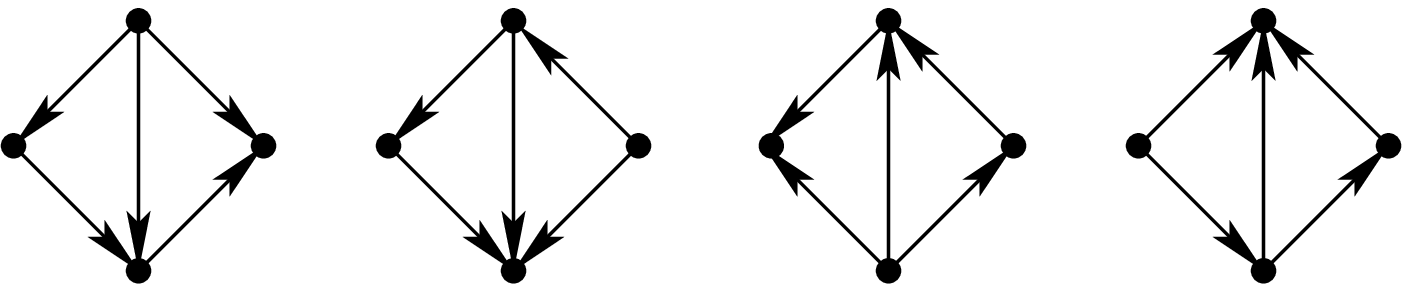,width=65mm}}

\subfigure{\psfig{figure=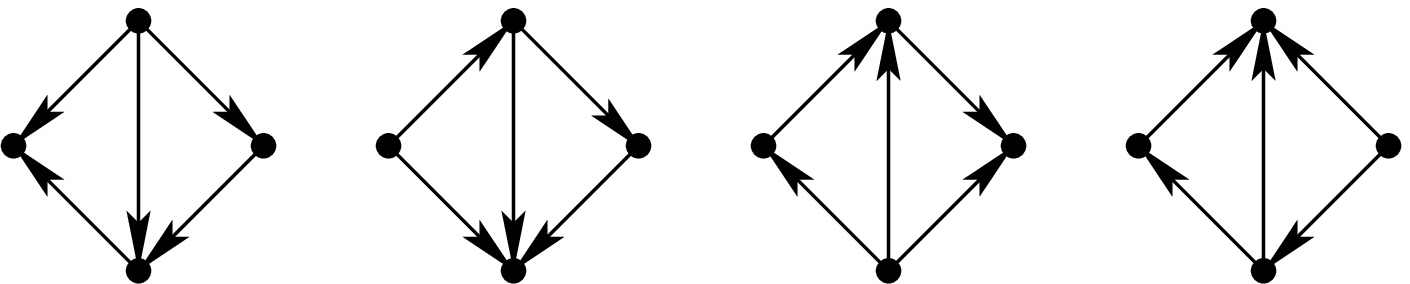,width=65mm}}

\subfigure{\psfig{figure=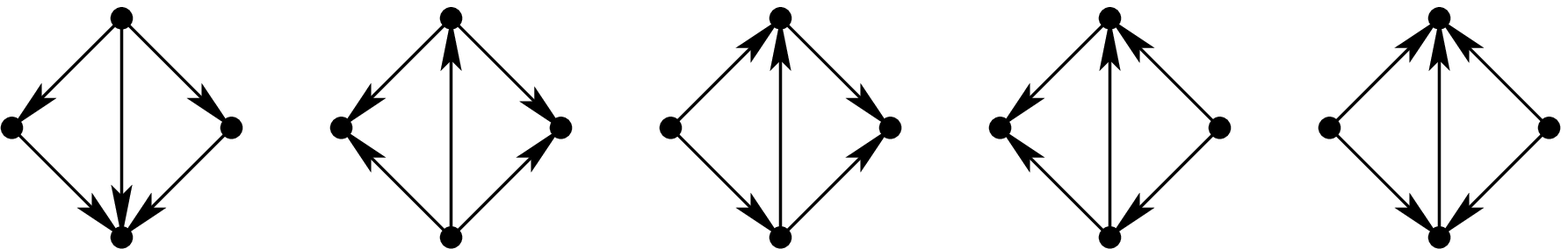,width=80mm}}

\subfigure{\psfig{figure=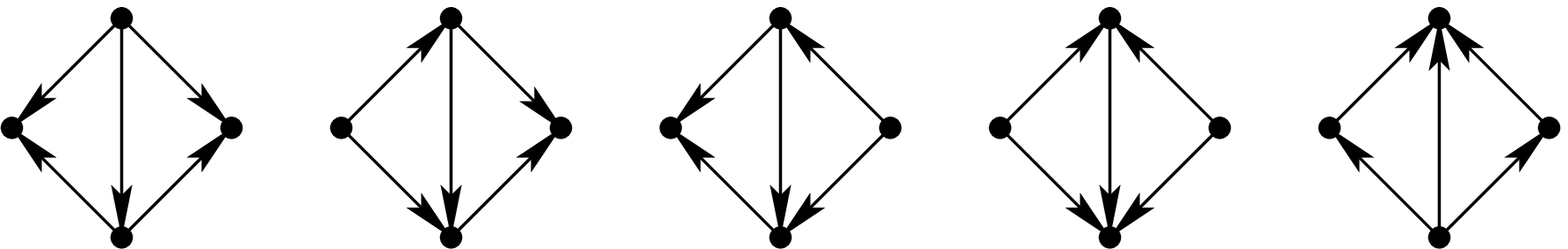,width=80mm}}
\caption{Cut-equivalence classes of acyclic orientations.
\label{Double-tiangle-acyclic}}
\end{figure}

There are total of 32 orientations. The other 14 cyclic orientations constitute
10 cut-equivalence classes, 6 of them are singletons, and each of the other 4
classes contains exactly two orientations.
\end{exmp}

{\bf Acknowledgement} The author thanks Prof. Libo Yang for reading
the manuscript carefully and offering some valuable comments.

\bibliographystyle{amsalpha}

\end{document}